\documentclass[11pt,epsfig,amsfonts]{amsart}
\usepackage{stmaryrd}
\usepackage{xcolor}
\usepackage{epsfig}
\usepackage{enumerate}
\usepackage{amsmath}
\usepackage{amssymb}
\usepackage{amscd}
\usepackage{graphicx}
\usepackage{lineno}
\usepackage{pstricks}
\usepackage{tocvsec2}
\usepackage{mathtools}
\allowdisplaybreaks[4]
\usepackage{setspace}

\usepackage [backref, colorlinks, citecolor=red, anchorcolor=black, linkcolor=blue]{hyperref}

\usepackage{amsfonts}
\usepackage[top=35mm, bottom=35mm, left=30mm, right=30mm]{geometry}
\usepackage{verbatim}
\usepackage[normalem]{ulem}
\usepackage{graphicx, amssymb, amsmath,amsthm}
\allowdisplaybreaks[4]
\usepackage{appendix}
 \usepackage{mathrsfs}
 \usepackage{cite}

\usepackage{setspace}
 \makeatletter
 \@namedef{subjclassname@2020}{%
 	\textup{2020} Mathematics Subject Classification}
 \makeatother

\usepackage{xcolor}

\usepackage{xcolor}

\newtheorem{definition}{Definition}[section]
\newtheorem{lemma}{Lemma}[section]
\newtheorem{theorem}{Theorem}
\newtheorem{corollary}{Corollary}[section]
\newtheorem{remark}{Remark}[section]
\newtheorem{example}{Example}[section]

\numberwithin{equation}{section}
\usepackage{epsfig}

\numberwithin{equation}{section}
\renewcommand*{\backref}[1]{}
\renewcommand*{\backrefalt}[4]{\quad \tiny
  \ifcase #1 (\textbf{NOT CITED.})%
  \or    (Cited in Section~#2.)%
  \else   (Cited in Section~#2.)%
  \fi}

\def \P{\mathbb{P}}
  \def \B{\mathcal{B}}
  \def \F{\mathcal{F}}
    
    \def \E{\mathcal{E}}
    \def \Prp{Pr_\P(\E)}
        \def \ip{\mathcal{I}_\P(\E)}
            \def \H{\mathcal{H}}
              \def \Y{\mathcal{Y}}
                \def \O{\mathcal{O}}
                  \def \L{\mathcal{L}}
                  \def \N{\mathcal{N}}
\makeatletter

\def\MRbibitem{\@ifnextchar[\my@lbibitem\my@bibitem}

\def\mybiblabel#1#2{\@biblabel{{\hyperref{http://www.ams.org/mathscinet-getitem?mr=#1}{}{}{#2}}}}

\def\myhyperanchor#1{\Hy@raisedlink{\hyper@anchorstart{cite.#1}\hyper@anchorend}}

\def\my@lbibitem[#1]#2#3#4\par{%
  \item[\mybiblabel{#2}{#1}\myhyperanchor{#3}\hfill]#4%
  \@ifundefined{ifbackrefparscan}{}{\BR@backref{#3}}%
  \if@filesw{\let\protect\noexpand\immediate
    \write\@auxout{\string\bibcite{#3}{#1}}}\fi\ignorespaces%
}

\def\my@bibitem#1#2#3\par{%
  \refstepcounter\@listctr
  \item[\mybiblabel{#1}{\the\value\@listctr}\myhyperanchor{#2}\hfill]#3%
  \@ifundefined{ifbackrefparscan}{}{\BR@backref{#2}}%
  \if@filesw\immediate\write\@auxout
    {\string\bibcite{#2}{\the\value\@listctr}}\fi\ignorespaces%
}

\makeatother


\author{Xue Liu}
\address[Xue Liu]
{School of Mathematics\\ Southeast University\\
 Nanjing 211189, P.R. China}
\email[X.~Liu]{xueliuseu@seu.edu.cn}

\author{Xiao Ma}
\address[Xiao Ma]
{School of Mathematical Sciences, University of Science and Technology of China, Hefei, Anhui 230026, P.R. China}
\email[X.~Ma]{xiaoma@ustc.edu.cn}

\author{Xiaomin Zhou}
\address[Xiaomin Zhou]
{School of Mathematics and Statistics, Huazhong University of Science and Technology, Wuhan, Hubei 430074, P.R. China}
\email[X.~Zhou]{zxm12@mail.ustc.edu.cn}

\begin{document}

\begin{abstract}
We study local stable and unstable sets for two-sided continuous bundle random dynamical
systems. For an ergodic invariant measure with positive fiber measure-theoretic entropy and positive fiberwise maximal Lyapunov exponent, we show that the fiber entropy is determined by the action of the random maps on local unstable sets, and establish a lower bound for the Hausdorff dimension of local unstable sets in terms of the ratio of entropy to the maximal fiber Lyapunov exponent. If the upper box dimension of the phase space is finite, we obtain a weak form of Ruelle's inequality.
\end{abstract}


\title[]{Local stable and unstable sets for random dynamical systems}
\maketitle

\parskip 10pt         

\section{Introduction}
\subsection{Background}
Random dynamical systems (RDSs) provide a natural framework for studying systems evolving in random environments, where both the admissible state space and the evolution maps may depend on the realization of the environment.
They are usually represented as skew products: a measure-preserving transformation describes the evolution of the noise, while a cocycle describes the evolution on the phase space.
This framework is particularly useful for the pathwise study of stochastic and randomly forced systems, where one wants to understand how randomness changes recurrence, instability, entropy production, and geometric complexity.
The basic framework and ergodic theory of RDSs were developed in the work of Arnold \cite{Arnold98}, Kifer \cite{Kifer1986,Kifer1988}, Liu and Qian \cite{liuqiansmooth95}, and Kifer and Liu \cite{KL06}, among others.

A central problem in dynamical systems is to relate a system's complexity to its geometric structure. Entropy quantifies the exponential growth of dynamically distinguishable orbit segments, either statistically with respect to an invariant measure or globally in the topological setting,
while Lyapunov exponents quantify the asymptotic exponential rates at which nearby trajectories separate or converge.
Their interaction is fundamental in both deterministic and random dynamics. In deterministic smooth dynamics, this relationship underlies the dimension theory of measure-preserving dynamical systems. The Kaplan--Yorke conjecture \cite{Kaplan-Yorke-79} and its subsequent development in the work of Frederickson, Kaplan, Yorke, and Yorke \cite{Frederickson1983}, Young's work on dimension, entropy, and Lyapunov exponents \cite{Young1982dim}, and the survey by Eckmann and Ruelle \cite{EckmannRuelle85} all emphasize that dimension is a fundamental quantity for describing chaotic behavior from a geometric viewpoint. In random smooth dynamics, Ledrappier and Young established entropy and dimension formulas for random transformations \cite{LedrYoung88,LedrappierYoung1988}, showing that randomness does not destroy this entropy--exponent--dimension mechanism but instead gives it a genuinely fiberwise form.

In smooth random hyperbolic theory, local stable and unstable manifolds are random smooth manifolds tangent to the stable and unstable Oseledets subspaces, respectively, and they collect points whose trajectories exhibit exponential contraction in forward or backward time along typical noise realizations.
They are central in random Pesin theory and depend on differentiability assumptions, typically $C^{1+\alpha}$ regularity of the fiber maps together with suitable integrability conditions; see \cite{Arnold98,liuqiansmooth95,Liu2001,KL06}.
Results of Ledrappier and Young \cite{LedrappierYoung1985} in the deterministic setting and Bahnm\"uller and Liu \cite{BahnmullerLiu1998} in the random setting showed that measure-theoretic entropy is determined by the action of the random maps on unstable manifolds. For merely continuous RDSs, invariant manifolds are no longer available. Nevertheless, entropy and metric expansion rates remain meaningful. This leads to the main questions of the present paper:
\begin{itemize}
  \item \textit{Is entropy still reflected in the action of random maps on suitable unstable sets? }
  \item  \textit{Moreover, how are the dimensions of local stable and unstable sets related to entropy and the rates of expansion of random orbits, quantified by suitable Lyapunov-type exponents?}
\end{itemize}

The deterministic theory suggests that stable and unstable sets retain
substantial information about positive entropy even outside differentiable settings. Blanchard, Host, and Ruette \cite{Blanchard2002} proved that positive
topological entropy forces the existence of proper asymptotic pairs; moreover,
for an ergodic invariant measure with positive entropy, almost every point
belongs to a non-trivial stable class, and in the invertible case such classes
contain uncountably many points forming Li--Yorke pairs for the inverse
dynamics. Subsequently, Huang, Xu, and Yi \cite{HXY2015} extended these results to positive
entropy $G$-systems, where  $G$ is a countable, discrete, infinite left-orderable amenable group.
Collectively, these works characterize chaotic behavior inside
closures of stable and unstable sets in positive-entropy systems \cite{Sumi2003,Huang2008,FangHuangYiZhang2012,HLY2014}.
Most directly related to our work, Feng, Gao, Huang, and Lian \cite{Fengscichina2022} proved that, for deterministic positive-entropy systems, local stable and unstable sets admit lower bounds on their Hausdorff dimensions in terms of entropy and the maximal Lyapunov exponent. These results indicate that stable and unstable sets retain substantial geometric information even when one works beyond uniformly hyperbolic settings or without a manifold structure on the phase space.

In this paper, we consider two-sided continuous bundle RDSs. 
 For every ergodic invariant measure with positive entropy, we prove that the fiber entropy is
determined by the action of the random maps on the unstable sets (see Theorem \ref{thm local entropy}). Moreover, we prove that the Hausdorff dimension of fiberwise local unstable sets is bounded below by the ratio of entropy to this exponent (see Theorem \ref{Thm main}). If the upper pointwise dimension of the conditional measure is finite, we obtain a weak form of Ruelle's inequality (see Corollary \ref{corollary ruelle}). Thus, the paper connects the entropy theory of RDSs with the metric geometry of local stable and unstable sets and extends the entropy--Lyapunov--dimension philosophy to a nonsmooth random setting.

\subsection{Settings and main results}
Let $(\Omega,\mathcal F)$ be a standard Borel space and let $\mathbb P$ be a probability measure on $(\Omega,\mathcal F)$. Let $\theta:\Omega\to\Omega$ be an invertible bimeasurable, $\mathbb P$-preserving ergodic transformation. 
Let $(X,\B)$ be a compact metric space with the Borel $\sigma$-algebra $\B$, and let $\E\subset \Omega\times X$ be a measurable random compact set; that is, $\E_\omega$ is a nonempty compact set for every $\omega\in\Omega$, and
$\E$ has measurable graph. We denote $\E_\omega:=\{x\in X:\ (\omega,x)\in \E\}$.
\begin{definition}
  A two-sided continuous bundle RDS $F$ over $(\Omega,\F,\P)$ on $\E\subset \Omega\times X$ is a family $F:=\{F_\omega^n:\E_\omega\rightarrow \E_{\theta^n\omega}\}_{\omega\in\Omega,n\in\mathbb{Z}}$ of homeomorphisms satisfying:
\begin{enumerate}
  \item $(\omega,x)\mapsto F_\omega^n(x)$ is jointly measurable from $\E$ to $X$ for all $n\in\mathbb{Z}$,
  \item $F_\omega^0=id$ for all $\omega\in\Omega$,
  \item the cocycle property: $F_\omega^{n+m}=F_{\theta^m\omega}^n\circ F_\omega^m$ for all $n,m\in\mathbb{Z}$, $\omega\in\Omega$.
\end{enumerate}
We simply denote $F_\omega^1$ by $F_\omega$. The skew-product transformation $\Theta:\E\to \E$ generated by $F$ and $\theta$ is defined by $\Theta(\omega,x)=(\theta\omega,F_\omega(x))$ and is invertible on $\E$.
\end{definition}

Continuous RDSs on compact metric spaces provide a natural topological framework for random dynamics beyond the differentiable setting.
Typical examples include stationary ergodic compositions of homeomorphisms of compact spaces,
random subshifts and other random symbolic systems, and random interval and circle homeomorphisms.
Moreover, after passing to their time-one maps, random flows generated by random differential equations and Stratonovich stochastic differential equations on compact manifolds give important classes of two-sided $C^1$ RDSs covered by our framework; see Theorems 2.2.14 and 4.2.10 for RDEs and Theorems 2.3.42 and 4.2.12 for SDEs in \cite{Arnold98}.

Let $\Prp$ denote the space of probability measures on $\Omega\times X$ with marginal $\P$ on $\Omega$ and supported on $\E$. We denote
\begin{equation}\label{eq def FE}
  \F_\E=\sigma\{(A\times X)\cap \E:\ A\in\F\}.
\end{equation}
By Rokhlin's disintegration theorem (see Lemma \ref{lemma measure disintegration}), every $\mu\in\Prp$ admits a disintegration $d\mu(\omega,x)=d\mu_\omega(x)d\P(\omega)$, where for $\P$-a.e. $\omega\in\Omega$, $\mu_\omega$ is a probability measure on the fiber $\E_\omega$, and the family $\{\mu_\omega\}_{\omega\in \Omega}$ is the collection of conditional measures with respect to the base $\sigma$-algebra $\F_\E$.
We denote by $\ip\subset\Prp$ the space of $\Theta$-invariant measures. For $\mu\in\ip$,
the invariance is equivalent to the equivariance relation
$
(F_\omega)_*\mu_{\omega}=\mu_{\theta\omega}$
for $\P$-a.e. $\omega\in\Omega$, where $(F_\omega)_{*}\mu_\omega(B)=\mu_\omega((F_\omega)^{-1}B)$ for any Borel measurable set $B\in \mathcal{B}(\E_{\theta\omega})$.
Denote by $\mathcal{I}_\P^e(\E)\subset \Prp$ the space of $\Theta$-invariant ergodic measures.

Before stating the main results, we introduce a metric expansion quantity that plays the role of the derivative growth rate in the continuous setting. The theory of {\it characteristic exponents} for random homeomorphisms of metric spaces, developed by Kifer \cite{Kifer1983LyapunovExponents, Kifer1984}, provides a fiberwise analogue of the maximal Lyapunov exponent.
For $n\geq 1$ and $\delta>0$, set
\begin{equation*}
  \lambda_n^\delta(\omega,x)=\sup_{y\in B_\omega(x,\delta,n),y\not=x}\frac{d(F_\omega^n(x),F_\omega^n(y))}{d(x,y)},
\end{equation*}
where $B_\omega(x,\delta,n)=\{y\in\E_\omega:\ \max_{0\leq i\leq n-1}d(F_\omega^i(x),F_\omega^i(y))<\delta\}.$
When $n=1$, $\lambda_1^\delta(\omega,x)$ is the local Lipschitz constant of $F_\omega$ near $x$ at scale $\delta$.
For $\delta>0$, the mapping $(\omega,x)\mapsto \lambda_n^\delta(\omega,x)$ is $(\F\otimes \B)\vert_{\E}$-measurable. Moreover, the sequence of functions $\{\log^+\lambda_n^\delta\}_{n\in\mathbb{N}}$ is subadditive with respect to the dynamical system $(\E, F)$, where $\log^{+} \lambda= \max\{0, \log \lambda\}$ for $\lambda>0$ and $\log^{+}0=0$; see Kifer \cite{Kifer1984}.
By Kingman's subadditive ergodic theorem, if $\log^+\lambda_1^\delta\in L^1(\mu)$ for some $\delta>0$, then the limit $\lim_{n\to+\infty}\frac{1}{n}\log^+\lambda_n^\delta(\omega,x)$ exists for $\mu$-a.e. $(\omega,  x)\in\E$ and the convergence holds in $L^1(\mu)$. Moreover, the resulting function is $\Theta$-invariant. Denote this limit by
        \begin{equation}\label{eq def lambda delta}
          \Lambda^\delta(\omega,x):=\lim_{n\to+\infty}\frac{1}{n}\log^+\lambda_n^\delta(\omega,x)\quad \mu\text{-a.e. } (\omega,x)\in\E.
        \end{equation}
For each \(n \geq 1\) and \((\omega,x)\in \E\), the map
\(\delta \mapsto \lambda_n^\delta(\omega,x)\) is monotone increasing in $\delta$ and bounded below by $0$.
Therefore, the  limit
$$\lambda_n(\omega,x):=\lim_{\delta\searrow 0}\lambda_n^\delta(\omega,x)$$
exists.
The sequence $\{\log^+\lambda_n\}_{n\in\mathbb{N}}$ is still subadditive, since it is the pointwise limit of subadditive sequences $\{\log^+\lambda_n^\delta\}_{n\in\mathbb N}$.
Moreover, since
$\lambda_1(\omega,x)\leq \lambda_1^{\delta^*}(\omega,x)$,
the assumption $\log^+\lambda_1^{\delta^*}\in L^1(\mu)$ for some $\delta^*>0$ implies that $\log^+\lambda_1\in L^1(\mu)$.
Hence, by Kingman's subadditive ergodic theorem,
\begin{equation}\label{eq def fiber Lyapunov exponent}
   \chi(\omega,x):=\lim_{n\to+\infty}\frac{1}{n}\log^+\lambda_n(\omega,x)
\end{equation}
exists for $\mu$-a.e. $(\omega,x)\in\E$, and the convergence holds in $L^1(\mu)$. We call $\chi(\omega,x)$ the {\bf fiberwise maximal Lyapunov exponent at $(\omega,x)$}.
In Theorem \ref{thm 3}, we show that this quantity agrees with the classical maximal Lyapunov exponent for $C^1$ RDSs under the usual integrability assumptions. When $\chi(\omega,x)$ exists $\mu$-a.e. and $\chi\in L^1(\mu)$, we define the maximal fiber Lyapunov exponent of the bundle RDS by
\begin{equation}\label{eq def lambda mu F}
  \chi_\mu(F):=\int_\E \chi(\omega,x) d\mu(\omega,x)=\lim_{n\to+\infty}\frac{1}{n}\int_\E\log^+\lambda_n(\omega,x)d\mu(\omega,x).
\end{equation}
Note that when $\mu$ is ergodic, $\chi_\mu(F)=\chi(\omega,x)$ for $\mu$-a.e. $(\omega,x)\in\E$.

We now define the fiberwise local stable and unstable sets. For a two-sided continuous bundle RDS $(\E,F)$, given $(\omega,x)\in\E$ and $\delta>0$, the $\delta$-stable set of $x$ in the fiber $\E_\omega$ is
\begin{equation*}
  W_\omega^s(x,\delta):=\{y\in\E_\omega: d(F_\omega^n(x),F_\omega^n(y))\leq \delta,\ \forall n\geq 0\mbox{ and }\lim_{n\to+\infty}d(F_\omega^n(x),F_\omega^n(y))=0\}.
\end{equation*}
Similarly, the $\delta$-unstable set of $x$ in the fiber $\E_\omega$ is
\begin{equation*}
  W_\omega^u(x,\delta):=\{y\in\E_\omega: d(F_\omega^{-n}(x),F_\omega^{-n}(y))\leq \delta,\ \forall n\geq 0\mbox{ and }\lim_{n\to+\infty}d(F_\omega^{-n}(x),F_\omega^{-n}(y))=0\}.
\end{equation*}
The following theorem establishes a lower bound for the Hausdorff dimension of local unstable sets in terms of the ratio of entropy to the maximal fiber Lyapunov exponent.
\begin{theorem}\label{Thm main}
  Let $(\E,F)$ be a two-sided continuous bundle RDS and let $\mu\in\mathcal{I}_\P^e(\E)$ with positive fiber entropy $h_\mu(F)>0$. Assume $\log^+\lambda_1^{\delta^*}\in L^1(\mu)$ for some $\delta^*>0$ and $\chi_\mu(F)\in(0,+\infty)$. Then for any $\delta>0$, we have
  \begin{equation*}
    \dim_H(W_\omega^u(x,\delta))\geq \frac{h_\mu(F)}{\chi_\mu(F)},\ \mbox{ for $\mu$-a.e. $(\omega,x)\in\E$},
  \end{equation*}
where the definition of fiber entropy $h_\mu(F)$ can be found in Section \ref{sec 3.3}.
If the assumptions on $(\E,F)$ are replaced by the corresponding assumptions on $(\E,F^{-1})$, then
\begin{equation*}
   \dim_H(W_\omega^s(x,\delta))\geq \frac{h_\mu(F^{-1})}{\chi_\mu(F^{-1})},\ \mbox{ for $\mu$-a.e. $(\omega,x)\in\E$}.
\end{equation*}
\end{theorem}

The main ingredient in the proof of Theorem \ref{Thm main} is the measurable partition constructed in Theorem \ref{thm local entropy}, which is subordinate to the fiberwise local unstable sets. This partition has the following properties: after the base $\sigma$-algebra $\F_\E$ is adjoined (see \eqref{eq def FE}), the corresponding conditional measures are concentrated on local unstable sets, and their fiberwise local entropies coincide with $h_\mu(F)$. This partition plays a role analogous to that of partitions subordinate to unstable manifolds; see Section 3 in \cite{LedrappierYoung1985} for the deterministic case and \cite[Proposition 3.7]{BahnmullerLiu1998} for the random setting.

\begin{theorem}\label{thm local entropy}
  Let $(\E,F)$ be a two-sided continuous bundle RDS, and let $\mu\in \mathcal{I}_\P^e(\E)$ with $h_\mu(F)>0$. Then for any $\delta>0$, there exists a countably generated measurable partition $\xi$ of $\E$ with the following properties:
  \begin{enumerate}
    \item $\Theta\xi\prec \xi$, $\overline{\xi_\omega(x)}\subset W_\omega^u(x,\delta)$ for each $x\in\E_\omega$, where $\xi_\omega(x)=\{y\in\E_\omega:\ (\omega,y)\in \xi(\omega,x)\}$ and $\xi(\omega,x)$ is the atom of $\xi$ containing $(\omega,x)$,
    \item Let $(\mu_\omega)_x$ denote the disintegration of $\mu$ over $\sigma(\xi\vee\F_\E)$, defined on the $\mu$-full measure set provided by Corollary \ref{corollary measure disintegration}. Then, for $\mu$-a.e. $(\omega,x)\in\E$, one has
        \begin{equation*}
          \underbar{h}_{(\mu_\omega)_x}(F,\omega,y)=\overline{h}_{(\mu_\omega)_x}(F,\omega,y)=h_\mu(F),\ \mbox{ for }(\mu_\omega)_x\mbox{-a.e. }y\in \E_\omega,
        \end{equation*}
where $\underbar{h}_{(\mu_\omega)_x}(F,\omega,y)$ and $\overline{h}_{(\mu_\omega)_x}(F,\omega,y)$ are the fiber lower and upper local entropies, respectively, of $F$ at $y\in \E_{\omega}$. The precise definitions of local entropies are given in Definition \ref{def local entropy}.
  \end{enumerate}
\end{theorem}

 The proof of Theorem \ref{Thm main} then follows the entropy-dimension strategy. The fiberwise maximal Lyapunov exponent compares ordinary balls in the fibers with fiberwise Bowen balls. Theorem \ref{thm local entropy} identifies the exponential decay of the conditional measures of Bowen balls with $h_\mu(F)$. Combining these estimates yields a lower bound for the lower pointwise dimension of the conditional measures, and the mass distribution principle gives the desired Hausdorff dimension estimate.

 We remark that, in Theorem \ref{Thm main}, positive entropy $h_\mu(F)$ does not rule out the possibility of $\chi_\mu(F)=0$; see Example 4.7 in \cite{Fengscichina2022}, which shows that entropy can be generated by complicated small-scale geometry without any positive metric expansion.  In smooth finite-dimensional dynamics, positive entropy implies the existence of positive Lyapunov exponents through the Margulis--Ruelle inequality.
 Nevertheless, the following theorem shows that if the conditional fiber measures have finite upper pointwise dimension and the fiberwise maximal Lyapunov exponent is zero, then $h_\mu(F)=0$.
 We also note that if the packing dimension (or the upper box dimension) of $X$ is finite, then condition \eqref{eq finite pointwise dimension} in the following theorem holds automatically; see Proposition 1.1 in \cite{carvalhoCondori2021}.
 \begin{theorem}\label{theorem 4}
   Let $(\E,F)$ be a continuous bundle RDS and let $\mu\in\mathcal{I}_\P^e(\E)$. Suppose that $\log^+\lambda_1^{\delta^*}\in L^1(\mu)$ for some $\delta^*>0$ and $\chi_\mu(F)\in[0,+\infty)$.  If
   \begin{equation}\label{eq finite pointwise dimension}
     \overline{d}_{\mu_\omega}(x):=\limsup_{r\searrow 0}\frac{\log \mu_\omega(B_\omega(x,r))}{\log r}<+\infty,\mbox{ for $\mu$-a.e. $(\omega,x)\in\E$},
   \end{equation}
   where $B_\omega(x,r):=\{y\in\E_\omega:\ d(x,y)<r\}$, then
   \begin{equation*}
     h_\mu(F)\leq \overline{d}_{\mu_\omega}(x)\cdot \chi_\mu(F),\mbox{ for $\mu$-a.e. $(\omega,x)\in\E$}.
   \end{equation*}
   In particular, under the finite upper pointwise dimension assumption \eqref{eq finite pointwise dimension}, $\chi_\mu(F)=0$ implies $h_\mu(F)=0$.
 \end{theorem}

As a corollary of Theorems \ref{Thm main} and \ref{theorem 4}, we obtain the following Ruelle-type inequality.
\begin{corollary}\label{corollary ruelle}
  Let $(\E,F)$ be a two-sided continuous bundle RDS and let $\mu\in\mathcal{I}_\P^e(\E)$. Assume $\log^+\lambda_1^{\delta^*}\in L^1(\mu)$ for some $\delta^*>0$. If $\overline{dim}_B(X)<+\infty$, then for any $\delta>0$, we have
  \begin{equation*}
    h_\mu(F)\leq \dim_H(W_\omega^u(x,\delta))\cdot \chi_\mu(F),\ \mbox{ for $\mu$-a.e. $(\omega,x)\in\E$},
  \end{equation*}
where $\overline{\dim}_{\mathrm B}(X)$ is the upper box dimension of $X$.
\end{corollary}
 The rest of the paper is organized as follows. Section \ref{section example} presents stationary RDSs generated by precompact families of homeomorphisms, including nonlinear examples on $\mathbb{T}^2$ with a positive fiberwise maximal Lyapunov exponent. Section \ref{sec preliminary lemma} collects preliminaries on conditional entropy, measure disintegration, fiber measure-theoretic entropy, and a local Shannon--McMillan--Breiman theorem for RDSs. Sections \ref{section proof thm 2}, \ref{section proof of thm main}, and \ref{sec thm 4} prove Theorems \ref{thm local entropy}, \ref{Thm main}, and \ref{theorem 4}, respectively. Section \ref{section Lyapunov exponents} proves that the fiberwise maximal Lyapunov exponent coincides with the classical maximal Lyapunov exponent for $C^1$ RDSs on smooth manifolds when the classical maximal Lyapunov exponent is positive, under suitable integrability assumptions.

\section{A concrete example with positive fiberwise maximal Lyapunov exponent}\label{section example}
The following standard facts can be found in \cite{Munkres}. Denote by $C(X,X)$ the space of continuous maps from $X$ to $X$ equipped with the compact-open topology $\mathcal{T}_{c.o.}$, which makes $C(X,X)$ a Polish space. The compact-open topology on $C(X,X)$ is generated by the subbasis
\begin{equation*}
  \mathcal{S}_{c.o.}=\{S(K,U)|\ K\subset X\mbox{ compact, }U\subset X\mbox{ open}\},
\end{equation*}
where $S(K,U)=\{f\in C(X,X):\ f(K)\subset U\}.$ 
Let $\operatorname{Homeo}(X,X)$ denote the Borel subset of $C(X,X)$ whose elements are homeomorphisms from $X$ to $X$. Given any Borel subset $\mathcal{U}\subset \operatorname{Homeo}(X,X)$, for instance, a subset of finite cardinality, we set
\begin{equation*}
  \Omega=\prod_{-\infty}^{+\infty}\mathcal{U}
\end{equation*}
equipped with the product $\sigma$-algebra. Then $\Omega$, endowed with the product $\sigma$-algebra, is a standard Lebesgue space.
Denote by $\theta$ the left shift, and let $\P$ be any $\theta$-invariant ergodic Borel probability measure. For each $\omega=(\dots,g_{-1}(\omega),g_0(\omega),g_1(\omega),\dots)$, where $\{g_i(\omega)\}_{i\in\mathbb{Z}}\subset \mathcal{U}$, we define
\begin{equation}\label{eqn-Fn}
  F_\omega^n=\begin{cases}
               g_{n-1}(\omega)\circ \cdots \circ g_0(\omega), & \mbox{if } n>0; \\
               id, & \mbox{if } n=0; \\
               g_n(\omega)^{-1}\circ \cdots \circ g_{-1}(\omega)^{-1}, & \mbox{if } n<0.
             \end{cases}
\end{equation}
Then $F$ defines a two-sided continuous RDS on $X$.

\begin{example}[A concrete example with positive fiberwise maximal Lyapunov exponent]\label{exmaple 1}
  Let $X=\mathbb{T}^2$ be equipped with the standard metric. Assume that $\psi_i:\mathbb{T}^1\to\mathbb{R}$ for $i\in\mathbb{Z}$ are Lipschitz continuous functions satisfying
   \begin{equation}\label{eq def L}
     L:=\sup\{Lip(\psi_i):\ i\in\mathbb{Z}\}\in(0,\frac{1}{2}).
   \end{equation}
We define the shear homeomorphisms $H_i\in \operatorname{Homeo}(\mathbb{T}^2,\mathbb{T}^2)$ for $i\in\mathbb{Z}$ by
  \begin{equation*}
    H_i(x,y)=(x,y+\psi_i(x)\ (\mbox{mod }1)),\ \forall (x,y)\in\mathbb{T}^2.
  \end{equation*}
Define $f_i:=H_i\circ T_A$ to be the composition of two homeomorphisms, where $T_A$ is the classical Arnold cat map, i.e.,
  \begin{equation*}
  T_A: \mathbb{T}^2\to \mathbb{T}^2,\quad
  \begin{pmatrix} x\\ y \end{pmatrix}\mapsto\begin{pmatrix}
        2 & 1 \\
        1 & 1
      \end{pmatrix} \begin{pmatrix} x\\ y \end{pmatrix};\end{equation*}
      and
    \begin{equation*} f_i(x,y)=(2x+y \ (\mbox{mod }1),x+y+\psi_i(2x+y) \ (\mbox{mod }1)),\ \forall (x,y)\in\mathbb{T}^2.
  \end{equation*}
Let $F$ be the two-sided continuous RDS generated as above by $\mathcal{U}=\{f_i:\ i\in\mathbb{Z}\}$.
It is straightforward to check that $\log^+\lambda_1^{\delta^*}\leq \log(\frac{3}{2}\cdot \frac{3+\sqrt{5}}{2})$ for any $\delta^*\in (0,1)$, where $\frac32$ is an upper bound for the Lipschitz constants of the maps $H_i$ and $\frac{3+\sqrt{5}}{2}$ is the maximal eigenvalue of $T_A$. Therefore, $\log^+\lambda_1^{\delta^*}\in L^1(\mu)$, and the fiberwise maximal Lyapunov exponent exists $\mu$-a.e. and in $L^1(\mu)$.

\begin{lemma}\label{lemma in example 1}
For any $\mu\in \mathcal{I}_\P^e(\Omega\times\mathbb{T}^2)$, the fiberwise maximal Lyapunov exponent in Example \ref{exmaple 1} satisfies
$$
\chi(\omega,z)\geq
\log\left(\frac{2}{\sqrt{1+(L+1)^2}}\right)>0,
\quad\mbox{for $\mu$-a.e. }(\omega,z)\in\Omega\times\mathbb{T}^2.
$$
\end{lemma}
\begin{proof}[Proof of Lemma \ref{lemma in example 1}]
The proof is based on a metric version of the cone criterion.
We choose $r\in(0,1/4)$ such that for any $z\in\mathbb T^2$ and any fixed lift
$\widetilde z\in\mathbb R^2$ of $z$, if $d(z,z')<r$, then there exists a unique lift
$\widetilde z'\in\mathbb R^2$ of $z'$ satisfying
$$
\|\widetilde z'-\widetilde z\|<r.
$$
For any $z\in\mathbb T^2$ and $0<s<r_0$, define the cone sector
$$
\operatorname{Cone}_K(z,s)
:=
\{\pi(\widetilde z+v):v\in K,\ \|v\|<s\},
$$
where $\pi:\mathbb R^2\to\mathbb T^2$ is the canonical projection
and $\widetilde z$ is any lift of $z$.

Write
$$
\widetilde\psi_i:=\psi_i\circ\pi_1,
$$
where $\pi_1:\mathbb R\to\mathbb T^1$ is the canonical projection,
and use the natural lift
$$
\widetilde f_i(u,v)
=
\bigl(2u+v,u+v+\widetilde\psi_i(2u+v)\bigr)
$$
of $f_i$. Since the natural lift
$
\widetilde H_i(u,v)
=
\bigl(u,v+\widetilde\psi_i(u)\bigr)
$
satisfies
$
\operatorname{Lip}(\widetilde H_i)
\leq1+L<\frac32
$
and the operator norm of
$\left(\begin{smallmatrix}2&1\\1&1\end{smallmatrix}\right)$ is
$(3+\sqrt5)/2$, every $\widetilde f_i$ is
$\Lambda$-Lipschitz, where
$
\Lambda:=
\frac32\cdot\frac{3+\sqrt5}{2}>1.
$

We first prove the cone invariance. We claim that for each
$i\in\mathbb Z$, each $z=(x,y)\in\mathbb T^2$, and every
$0<\rho<r_0$,
\begin{equation}\label{eq:cone-invariance}
f_i\bigl(\operatorname{Cone}_K(z,\rho/\Lambda)\bigr)
\subset
\operatorname{Cone}_K(f_i(z),\rho).
\end{equation}
Indeed, let
$$
z'=(x',y')\in\operatorname{Cone}_K(z,\rho/\Lambda).
$$
Choose lifts so that
$$
\widetilde z'-\widetilde z
=(\Delta x,\Delta y)\in K,
\qquad
\|\widetilde z'-\widetilde z\|<\rho/\Lambda.
$$
Writing
$\widetilde z=(\widetilde x,\widetilde y)$ and
$\widetilde z'=(\widetilde x',\widetilde y')$, 
we have
$$
\widetilde f_i(\widetilde z')
-\widetilde f_i(\widetilde z)
=
(2\Delta x+\Delta y,\Delta x+\Delta y+\widetilde\psi_i(2\widetilde x'+\widetilde y')
-
\widetilde\psi_i(2\widetilde x+\widetilde y))
$$
and
$$
|\widetilde\psi_i(2\widetilde x'+\widetilde y')
-
\widetilde\psi_i(2\widetilde x+\widetilde y)|\leq L(2\Delta x+\Delta y).
$$
Consequently,
$$
\Delta x+\Delta y+\widetilde\psi_i(2\widetilde x'+\widetilde y')
-
\widetilde\psi_i(2\widetilde x+\widetilde y)
\geq
(1-2L)\Delta x+(1-L)\Delta y
\geq0,
$$
where the last inequality follows from $L<1/2$.
On the other hand,
$$
\begin{aligned}
\Delta x+\Delta y+\widetilde\psi_i(2\widetilde x'+\widetilde y')
-
\widetilde\psi_i(2\widetilde x+\widetilde y)
\leq
\Delta x+\Delta y+L(2\Delta x+\Delta y)
\leq
(L+1)(2\Delta x+\Delta y).
\end{aligned}
$$
Thus
$$
\widetilde f_i(\widetilde z')
-\widetilde f_i(\widetilde z)\in K.
$$
Moreover,
$$
\|\widetilde f_i(\widetilde z')
-\widetilde f_i(\widetilde z)\|
\leq
\Lambda\|\widetilde z'-\widetilde z\|
<\rho.
$$
Therefore, \eqref{eq:cone-invariance} holds.

Next, we establish the expansion estimate inside the cone. For every
$0<\rho<r_0$ and every
$
z'\in
\operatorname{Cone}_K(z,\rho/\Lambda)\setminus\{z\},
$
we have
\begin{equation}\label{eq:cone-expansion}
\frac{d(f_i(z'),f_i(z))}{d(z',z)}
\geq
\frac{2\Delta x+\Delta y}
{\sqrt{(\Delta x)^2+(\Delta y)^2}}
\geq
\frac{2}{\sqrt{1+(L+1)^2}}
=:c.
\end{equation}
Indeed,
$$
\|\widetilde z'-\widetilde z\|
<\rho/\Lambda<r_0
\text{ and }
\|\widetilde f_i(\widetilde z')
-\widetilde f_i(\widetilde z)\|
<\rho<r_0.
$$
Therefore, the input and output displacements are the unique short
lifts, and hence
$$
d(z',z)=\|\widetilde z'-\widetilde z\| \text{ and }
d(f_i(z'),f_i(z))
=
\|\widetilde f_i(\widetilde z')
-\widetilde f_i(\widetilde z)\|.
$$
Since $z'\neq z$ and
$(\Delta x,\Delta y)\in K$, we have $\Delta x>0$. Moreover,
$$
\sqrt{(\Delta x)^2+(\Delta y)^2}
\leq
\sqrt{1+(L+1)^2}\,\Delta x
\text{ and }
2\Delta x+\Delta y\geq2\Delta x.
$$
This proves \eqref{eq:cone-expansion}. Since $L\in(0,1/2)$,
we have $c>1$.

Now fix $(\omega,z)\in\Omega\times\mathbb T^2$ such that
$\chi(\omega,z)$ is well-defined, and fix $\delta>0$. Choose
$$
0<r<\min\{\delta,r_0\}.
$$
For a fixed $n\geq1$, put
$$
\rho_k:=
\frac{r}{\Lambda^{n-k-1}},
\qquad 0\leq k\leq n-1.
$$
Then
$$
0<\rho_k\leq r<r_0
\qquad\text{and}\qquad
\frac{\rho_k}{\Lambda}
=
\frac{r}{\Lambda^{n-k}}.
$$
Repeatedly applying \eqref{eq:cone-invariance}, we obtain
$$
F_\omega^k
\bigl(\operatorname{Cone}_K(z,r/\Lambda^n)\bigr)
\subset
\operatorname{Cone}_K
\bigl(F_\omega^k(z),r/\Lambda^{n-k}\bigr),
\qquad
0\leq k\leq n,
$$
where $F^k_{\omega}$ is defined in \eqref{eqn-Fn}.
For $0\leq k\leq n-1$, every point in the latter cone lies at
distance less than $r<\delta$ from $F_\omega^k(z)$. Hence
$$
\operatorname{Cone}_K(z,r/\Lambda^n)
\subset
B_\omega(z,\delta,n).
$$
For any
$
z'\in
\operatorname{Cone}_K(z,r/\Lambda^n)\setminus\{z\},
$
applying \eqref{eq:cone-expansion} with $\rho=\rho_k$ at the
$k$th step gives
$$
\frac{d(F_\omega^n(z),F_\omega^n(z'))}{d(z,z')}
\geq c^n.
$$
Consequently,
$$
\lambda_n^\delta(\omega,z)\geq c^n \text{ for every $n\geq1$.}
$$
 Since $\delta>0$ was arbitrary, letting
$\delta\searrow0$ in the definition of $\lambda_n(\omega,z)$ gives
$$
\lambda_n(\omega,z)
=
\lim_{\delta\searrow0}\lambda_n^\delta(\omega,z)
\geq c^n.
$$
Therefore,
$$
\begin{aligned}
\chi(\omega,z)
&=
\lim_{n\to+\infty}
\frac1n\log^+\lambda_n(\omega,z)\geq
\log c
=\log\left(
\frac{2}{\sqrt{1+(L+1)^2}}\right)>0.
\end{aligned}
$$
Since $\chi(\omega,z)$ is defined for $\mu$-a.e.
$(\omega,z)$, the conclusion follows.
\end{proof}

\end{example}

\section{Preliminaries}\label{sec preliminary lemma}

\subsection{Measure-theoretic entropy}
In this section, we introduce the notation and lemmas on measure-theoretic entropy that will be used below. Let $(Y,\mathcal{Y},\nu)$ be a measure space, and let $\alpha=\{A_j\}_{j\in J}$ be a finite or countable measurable partition. Let $\mathcal{C}\subset \mathcal{Y}$ be a $\sigma$-algebra. The conditional information function and the conditional entropy of $\alpha$ with respect to $\mathcal{C}$ are defined, respectively, by
\begin{equation}\label{eq def of information}
  I_\nu(\alpha|\mathcal{C})(x)=-\sum_{A\in\alpha}1_A(x)\log E_\nu(1_A|\mathcal{C})(x)
\end{equation}
and
\begin{equation*}
  H_\nu(\alpha|\mathcal{C}):=\int_Y I_\nu(\alpha|\mathcal{C})d\nu=\sum_{A\in\mathcal{\alpha}}\int_Y-E_\nu(1_A|\mathcal{C})\log E_\nu(1_A|\mathcal{C})d\nu,
\end{equation*}
where $E_\nu(1_A|\mathcal{C})$ is the conditional expectation of $1_A$ with respect to $\mathcal{C}$. For a measurable partition $\beta$, let $\sigma(\beta)$ be the $\sigma$-algebra generated by $\beta$.
\begin{remark}
  When convenient, we will use $\beta$ to denote $\sigma(\beta)$; thus, we write $I_\nu(\alpha|\beta) $ and $H_\nu(\alpha|\beta)$ for $I_\nu(\alpha|\sigma(\beta)) $ and $H_\nu(\alpha|\sigma(\beta))$, respectively, and $\alpha\vee \mathcal{C} $ for $\sigma(\alpha)\vee \mathcal{C}$, etc.
\end{remark}

We collect several necessary facts in the following lemma, which can be found in \cite[Proposition 14.16, Lemma 14.27, and Theorem 14.28]{Glasner} except for (2).
\begin{lemma}\label{lemma four useful lemma}
 Let $(Y,\mathcal{Y},\nu)$ be a probability space. For countable measurable partitions $\alpha$, $\beta,\ \gamma$, and a $\sigma$-algebra $\mathcal{C}\subset\mathcal{Y}$, one has
  \begin{enumerate}
    \item The information cocycle equation: $I_\nu(\alpha\vee \beta|\mathcal{C})(y)=I_\nu(\alpha|\mathcal{C})(y)+I_\nu(\beta|\alpha\vee \mathcal{C})(y)$ for $\nu$-a.e. $y\in Y$.
    \item If $T:(Y,\mathcal{Y})\to (Y,\mathcal{Y})$ is a measurable mapping, then $I_\nu(T^{-1}\alpha|T^{-1}\mathcal{C})(x)=I_{T_*\nu}(\alpha|\mathcal{C})(Tx)$ for $\nu$-a.e. $x\in Y$.
    \item The entropy cocycle equation: $H_\nu(\alpha\vee\beta|\mathcal{C})=H_\nu(\alpha|\mathcal{C})+H_\nu(\beta|\alpha\vee\mathcal{C}).$
    \item Chung's lemma: if $\mathcal{C}_1\subset\mathcal{C}_2\subset\cdots $ is an increasing sequence of $\sigma$-algebras with $\mathcal{C}_n\nearrow\mathcal{C}$ (or a decreasing sequence with $\mathcal{C}_n\searrow\mathcal{C}$), and $H_\nu(\alpha)<\infty$, then
        \begin{equation*}
          f:=\sup_{n}I_\nu(\alpha|\mathcal{C}_n)\in L^1(\nu),\mbox{ and }\int_Yfd\nu\leq H_\nu(\alpha)+1.
        \end{equation*}
        \item The martingale convergence theorem for entropy: under the same conditions as in Chung's lemma, one has $I_\nu(\alpha|\mathcal{C}_n)\to I_\nu(\alpha|\mathcal{C})$ $\nu$-a.e. and in $L^1(\nu)$; moreover, $H_\nu(\alpha|\mathcal{C}_n)\searrow H_{\nu}(\alpha|\mathcal{C})$ (or $H_\nu(\alpha|\mathcal{C}_n)\nearrow H_{\nu}(\alpha|\mathcal{C})$, respectively).
  \end{enumerate}
\end{lemma}
\begin{proof}[Proof of Lemma \ref{lemma four useful lemma}]
To prove part (2), it suffices to establish, for every $A\in\alpha$, that
\begin{equation}\label{eq conditional expectation pullback}
E_\nu(1_{T^{-1}A}\mid T^{-1}\mathcal C)(x)
=
E_{T_*\nu}(1_A\mid\mathcal C)(Tx)
\end{equation}
for $\nu$-a.e. $x\in Y$.

Fix $A\in\alpha$ and define
$
g_A(x):=E_{T_*\nu}(1_A\mid\mathcal C)(Tx).
$
Since $E_{T_*\nu}(1_A\mid\mathcal C)$ is $\mathcal C$-measurable, the function $g_A$ is $T^{-1}\mathcal C$-measurable.

For every $B\in\mathcal C$, the defining property of conditional expectation gives
\begin{equation*}
\begin{aligned}
\int_{T^{-1}B}
E_\nu(1_{T^{-1}A}\mid T^{-1}\mathcal C)(x)\,d\nu(x)=\int_{T^{-1}B}1_{T^{-1}A}(x)\,d\nu(x)=
T_*\nu(A\cap B).
\end{aligned}
\end{equation*}
On the other hand, by the definition of the pushforward measure,
\begin{equation*}
\begin{aligned}
\int_{T^{-1}B}g_A(x)\,d\nu(x)
&=
\int_{T^{-1}B}
E_{T_*\nu}(1_A\mid\mathcal C)(Tx)\,d\nu(x)=
\int_B E_{T_*\nu}(1_A\mid\mathcal C)(y)\,dT_*\nu(y)\\
&=
\int_B1_A(y)\,dT_*\nu(y)=
T_*\nu(A\cap B).
\end{aligned}
\end{equation*}
Since every element of $T^{-1}\mathcal C$ is of the form $T^{-1}B$ for some $B\in\mathcal C$, the uniqueness of conditional expectation implies \eqref{eq conditional expectation pullback}.

Because $\alpha$ is countable, we may choose a common $\nu$-full measure set on which \eqref{eq conditional expectation pullback} holds for every $A\in\alpha$. Therefore,
\begin{equation*}
\begin{aligned}
I_\nu(T^{-1}\alpha\mid T^{-1}\mathcal C)(x)
&=
-\sum_{A\in\alpha}
1_{T^{-1}A}(x)
\log E_\nu(1_{T^{-1}A}\mid T^{-1}\mathcal C)(x)\\
&=
-\sum_{A\in\alpha}
1_A(Tx)
\log E_{T_*\nu}(1_A\mid\mathcal C)(Tx)\\
&=
I_{T_*\nu}(\alpha\mid\mathcal C)(Tx)
\end{aligned}
\end{equation*}
for $\nu$-a.e. $x\in Y$. This proves part (2).
\end{proof}

\subsection{Double disintegration of measures}

The following theorem of measure disintegration can be found in \cite[Theorem 5.14]{Thomas2011}.
\begin{lemma}\label{lemma measure disintegration}
  Let $(Y,\mathcal{Y},\nu)$ be a Lebesgue space, and let $\mathcal{A}\subset\mathcal{Y}$ be a countably generated $\sigma$-algebra. Then there exist an $\mathcal{A}$-measurable full-measure set $Y^\prime\subset Y$ and a system $\{\nu_y^\mathcal{A}|\ y\in Y^\prime\}$ of measures on $Y$, referred to as the conditional measures, with the following properties:
  \begin{enumerate}
    \item For all $f\in L^1(Y,\mathcal{Y},\nu)$, one has
 $
      E(f|\mathcal{A})(y)=\int f(x)d\nu_y^\mathcal{A}(x),\ \nu\mbox{-a.e. }y\in Y.
$ Moreover, the above equation uniquely determines $\nu_y^\mathcal{A}$ for $\nu$-a.e. $y\in Y$.
    \item For all $y\in Y^\prime$, $\nu_y^\mathcal{A}([y]_\mathcal{A})=1$, where $[y]_\mathcal{A}=\cap_{y\in A\in\mathcal{A}}A$; moreover, $\nu_x^\mathcal{A}=\nu_y^\mathcal{A}$ for $x,y\in Y^\prime$ whenever $[x]_\mathcal{A}=[y]_{\mathcal{A}}.$
  \end{enumerate}
\end{lemma}
For a measurable set $A\subset \E$ and $\omega\in\Omega$, define its
fiber section by
$$
A_\omega:=\{x\in \E_\omega:(\omega,x)\in A\}.
$$
For a measurable partition $\alpha$ of $\E$, let
$$
\alpha_\omega
:=
\{A_\omega:A\in\alpha,\ A_\omega\neq\varnothing\}.
$$
Then $\alpha_\omega$ is a measurable partition of $\E_\omega$. Similarly,
if $\mathcal A$ is a sub-$\sigma$-algebra of
$(\mathcal F\otimes\mathcal B)|_\E$, define
$$
\mathcal A_\omega:=\{A_\omega:A\in\mathcal A\}.
$$
Then $\mathcal A_\omega$ is a sub-$\sigma$-algebra of
$\mathcal B(\E_\omega)$.
In this paper, we use the following double
disintegration of measures.

\begin{corollary}\label{corollary measure disintegration}
  Let $\xi$ be a countably generated measurable partition of $\E$ and let $\mu\in \Prp$. Then there exist a $\mu$-full measure set $\E^\prime$ and a uniquely determined family of regular conditional probability measures $(\mu_\omega)_x$ on $\E_\omega$ for each $(\omega,x)\in\E^\prime$ satisfying the following properties.
   \begin{enumerate}
    \item For all $f\in L^1(\mu)=L^1(\E,(\F\otimes \B)\vert_{\E},\mu)$, one has
        \begin{equation}\label{eq integral with respect to conditonal measure}
          E(f|\xi\vee \F_\E)(\omega,x):=E(f|\sigma(\xi\vee\F_\E))(\omega,x)=\int f(\omega,y)d(\mu_\omega)_x(y).
        \end{equation}
    \item For all $(\omega,x)\in\E^\prime$, $(\mu_\omega)_x( \E^\prime_\omega\cap\xi_\omega(x))=1$; moreover, for $(\omega,x),(\omega,y)\in\E^\prime$, $\xi_\omega(x)=\xi_\omega(y)$ implies that $(\mu_\omega)_x=(\mu_\omega)_y.$
    \item For $\P$-a.e. $\omega\in\Omega$, $\mu_\omega=\int_{\E_\omega}(\mu_\omega)_xd\mu_\omega(x)$, i.e. for any measurable set $B\subset\E_\omega$, $\mu_\omega(B)=\int_{\E_\omega}(\mu_\omega)_x(B)d\mu_\omega(x).$
  \end{enumerate}
In this case, for $\P$-a.e. $\omega\in\Omega$, $(\mu_\omega)_x$ can be viewed as the conditional measure of $\mu_\omega$ with respect to the countably generated measurable partition $\xi_\omega$.
\end{corollary}
\begin{proof}[Proof of Corollary \ref{corollary measure disintegration}]
Because $(\Omega,\mathcal F)$ is standard Borel, $\mathcal F$ is countably generated and separates points.
Consequently, $\mathcal F_{\mathcal E}$ is countably generated, and the atom of $\sigma(\xi\vee\mathcal F_{\mathcal E})$ containing $(\omega,x)$ is
$\{\omega\}\times\xi_\omega(x)=\{(\omega,y):\ y\in\xi_\omega(x)\}$.
  As a direct consequence of Lemma \ref{lemma measure disintegration}, there exist a $\mu$-full-measure set $\E^\prime$ and a uniquely determined family of regular conditional probability measures $\{(\mu_\omega)_x^*:\ (\omega,x)\in\E^\prime\}$ with respect to the $\sigma$-algebra $\sigma(\xi\vee \F_\E)$ satisfying the following properties.
  \begin{enumerate}
    \item[(a)] For all $f\in L^1(\mu)=L^1(\E,(\F\otimes \B)\vert_{\E},\mu)$, one has
        \begin{equation*}
          E(f|\xi\vee \F_\E)(\omega,x):=E(f|\sigma(\xi\vee\F_\E))(\omega,x)=\int f(\omega^\prime,x^\prime)d(\mu_\omega)_x^*(\omega^\prime,x^\prime).
        \end{equation*}
    \item[(b)] For all $(\omega,x)\in\E^\prime$, $(\mu_\omega)_x^*(\{\omega\}\times (\E^\prime_\omega\cap\xi_\omega(x)))=1$; moreover, for $(\omega,x),(\omega,y)\in\E^\prime$, $\xi_\omega(x)=\xi_\omega(y)$ implies that $(\mu_\omega)_x^*=(\mu_\omega)_y^*.$
  \end{enumerate}
  Note that $(\mu_\omega)_x^*$ is concentrated on the fiber $\{\omega\}\times \E_\omega$; we simply denote by $(\mu_\omega)_x$ the copy on $\E_\omega$ such that $(\mu_\omega)_x(B)=(\mu_\omega)_x^*(\{\omega\}\times B)$ for any measurable set $B\subset \E_\omega$. We can rewrite (a) and (b) to get (1) and (2).

  Finally, we prove (3) using the uniqueness of disintegration; see also \cite[Proposition 3.6]{Hans02}. For any bounded measurable function $f:\E\to \mathbb{R}$, a direct computation gives
  \begin{equation*}
    \begin{split}
        &\int_\Omega\int_{\E_\omega}\int_{\E_\omega}f(\omega,y)d(\mu_\omega)_x(y)d\mu_\omega(x)d\P(\omega) \overset{\eqref{eq integral with respect to conditonal         measure}}{=} \int_\Omega\int_{\E_\omega}E(f|\xi\vee \F_\E)(\omega,x)d\mu_\omega(x)d\P(\omega) \\
         =& \int_\E E(f|\xi\vee \F_\E)(\omega,x)d\mu(\omega,x)=\int_\E f(\omega,x)d\mu(\omega,x)=\int_\Omega\int_{\E_\omega}f(\omega,x)d\mu_\omega(x) d\P(\omega),
    \end{split}
  \end{equation*}
where the third equality follows from the property of conditional expectation.
By the uniqueness of the disintegration of $\mu$ over $\F_\E$, we conclude that $\mu_\omega=\int_{\E_\omega}(\mu_\omega)_xd\mu_\omega(x)$ for $\P$-a.e. $\omega\in\Omega$.
\end{proof}

\subsection{Fiber measure-theoretic entropy}\label{sec 3.3}
Recall that $\F_\E$ is defined in \eqref{eq def FE}.
For a measurable partition $\alpha$ of $\E$, and integers $m<n$, we denote
\begin{equation*}
  \alpha_m^n:=\bigvee_{k=m}^n\Theta^{-k}\alpha,\ \alpha^+:=\alpha_1^\infty=\bigvee_{k=1}^\infty\Theta^{-k}\alpha,\mbox{ and } \alpha^-:=\alpha_{-\infty}^{-1}=\bigvee_{k=1}^\infty\Theta^{k}\alpha.
\end{equation*}
For any $\mu\in \mathcal{I}_\P(\E)$,
the fiber measure-theoretic entropy of $F$ (or the relative entropy of $\Theta$) with respect to $\mu$ is defined by
\begin{equation*}
  h_\mu(F)=h_{\mu}^{(r)}(\Theta)=\sup_{\alpha}h_{\mu}(F,\alpha),
\end{equation*}
where
\begin{equation*}
  h_{\mu}(F,\alpha)=\lim_{n\rightarrow\infty}\frac{1}{n}H_\mu(\alpha_{0}^{n-1}|\F_\E),
\end{equation*}
and the supremum is taken over all finite or countable measurable partitions $\alpha$ of $\E$ with finite conditional entropy $H_\mu(\alpha|\F_\E)<\infty$.

The following lemma records a basic property that we use throughout this paper.
\begin{lemma}\cite[Lemma 2.2.3]{bogenschiitz1993equilibrium}
  For any two countable measurable partitions $\alpha$ and $\beta$ of $\E$ and $\mu\in \mathcal{I}_\P(\E)$, we have
  \begin{equation*}
    H_\mu(\alpha|\beta\vee\F_\E)=\int_\Omega H_{\mu_\omega}(\alpha_\omega|\beta_\omega)d\P(\omega).
  \end{equation*}
In particular, $H_\mu(\alpha|\F_\E)=\int H_{\mu_\omega}(\alpha_\omega)d\P(\omega)$.
\end{lemma}
\begin{lemma}\label{lemma fiber to global}
  Let $\beta$ be a countable measurable partition of $\E$, and let
$\mathcal C$ be a countably generated sub-\(\sigma\)-algebra of $(\F\otimes \B)\vert_{\E}$.
Then we have
\begin{equation}\label{eq fiber to global}
  I_\mu(\beta\mid \mathcal C\vee\mathcal F_\E)(\omega,x)
=
I_{\mu_\omega}(\beta_\omega\mid \mathcal C_\omega)(x)\mbox{ for $\mu$-a.e. $(\omega,x)\in\E$}.
\end{equation}
\end{lemma}
\begin{proof}[Proof of Lemma \ref{lemma fiber to global}]
Fix $A\in\beta$ and define
$$
\psi_A(\omega,x)
:=
E_\mu(1_A\mid\mathcal C\vee\mathcal F_\E)(\omega,x).
$$
Since $\psi_A$ is $\mathcal C\vee\mathcal F_\E$-measurable and  $
(\mathcal C\vee\mathcal F_\E)_\omega=\mathcal C_\omega$,
the section
$$
\psi_{A,\omega}(x):=\psi_A(\omega,x)
$$
is $\mathcal C_\omega$-measurable on $\E_\omega$ for
$\P$-a.e. $\omega\in\Omega$.

Since $\mathcal C$ is countably generated, there exists a countable
algebra $\mathcal C_0\subset\mathcal C$ such that
$
\sigma(\mathcal C_0)=\mathcal C.
$  Fix $B\in\mathcal C_0$. For every $D\in\mathcal F$, the defining
property of conditional expectation gives
\begin{equation}\label{eq int=int}
  \begin{split}
\int_D\int_{B_\omega}
\psi_{A,\omega}(x)\,d\mu_\omega(x)\,d\P(\omega)
&=
\int_{B\cap((D\times X)\cap\E)}
\psi_A(\omega,x)\,d\mu(\omega,x)\\
&=
\int_{B\cap((D\times X)\cap\E)}
1_A(\omega,x)\,d\mu(\omega,x)\\
&=
\int_D
\mu_\omega(A_\omega\cap B_\omega)\,d\P(\omega).
  \end{split}
\end{equation}
Since $D\in\mathcal F$ is arbitrary, it follows that
$
\int_{B_\omega}\psi_{A,\omega}(x)\,d\mu_\omega(x)
=
\mu_\omega(A_\omega\cap B_\omega)
$ for $\P$-a.e. $\omega\in\Omega$.

For each $B\in\mathcal C_0$, there exists a measurable set
$\Omega_{A,B}\subset\Omega$ with $\P(\Omega_{A,B})=1$ such that \eqref{eq int=int} holds for every $\omega\in\Omega_{A,B}$.
Since $\mathcal C_0$ is countable,
the set
$$
\Omega_A:=\bigcap_{B\in\mathcal C_0}\Omega_{A,B}
$$
satisfies $\P(\Omega_A)=1$. Hence \eqref{eq int=int} holds
simultaneously for every $B\in\mathcal C_0$ and every
$\omega\in\Omega_A$.

Fix $\omega\in\Omega_A$ and define
$$
\mathcal M_\omega
:=
\left\{
G\in\mathcal C_\omega:
\int_G\psi_{A,\omega}(x)\,d\mu_\omega(x)
=
\mu_\omega(A_\omega\cap G)
\right\}.
$$
By construction,
$$
(\mathcal C_0)_\omega
:=
\{B_\omega:B\in\mathcal C_0\}
\subset\mathcal M_\omega.
$$

We claim that $\mathcal M_\omega$ is a monotone class. Indeed, suppose
that $G_1\subset G_2\subset\cdots$ are elements of $\mathcal M_\omega$
and
$
G=\bigcup_{n=1}^{+\infty}G_n.
$
Since $\psi_{A,\omega}\geq0$, the monotone convergence theorem gives
$$
\begin{aligned}
\int_G\psi_{A,\omega}\,d\mu_\omega
&=
\lim_{n\to+\infty}
\int_{G_n}\psi_{A,\omega}\,d\mu_\omega
&=
\lim_{n\to+\infty}
\mu_\omega(A_\omega\cap G_n)
&=
\mu_\omega(A_\omega\cap G).
\end{aligned}
$$
Thus $G\in\mathcal M_\omega$.

Similarly, suppose that $G_1\supset G_2\supset\cdots$ are elements of
$\mathcal M_\omega$ and
$
G=\bigcap_{n=1}^{+\infty}G_n.
$
Since $0\leq\psi_{A,\omega}\leq1$, the dominated convergence theorem
gives
$$
\begin{aligned}
\int_G\psi_{A,\omega}\,d\mu_\omega
&=
\lim_{n\to+\infty}
\int_{G_n}\psi_{A,\omega}\,d\mu_\omega
&=
\lim_{n\to+\infty}
\mu_\omega(A_\omega\cap G_n)
&=
\mu_\omega(A_\omega\cap G).
\end{aligned}
$$
Hence $G\in\mathcal M_\omega$.

Therefore, $\mathcal M_\omega$ is a monotone class. Since taking fiber
sections commutes with complements and countable unions, one has
$$
\sigma\bigl((\mathcal C_0)_\omega\bigr)=\mathcal C_\omega.
$$
The monotone class theorem now yields
$
\mathcal M_\omega=\mathcal C_\omega.
$
Consequently,
$
\int_G\psi_{A,\omega}(x)\,d\mu_\omega(x)
=
\mu_\omega(A_\omega\cap G)
$
for every $G\in\mathcal C_\omega$.

Since $\psi_{A,\omega}$ is $\mathcal C_\omega$-measurable, the
definition and uniqueness of conditional expectation imply that
 for every $\omega\in\Omega_A$, the function
$\psi_{A,\omega}$ is a version of
$
E_{\mu_\omega}(1_{A_\omega}\mid\mathcal C_\omega).
$
We therefore choose a good version
\[
E_{\mu_\omega}(1_{A_\omega}\mid\mathcal C_\omega)
:=\psi_{A,\omega},
\qquad \omega\in\Omega_A,
\]
and define it arbitrarily for $\omega\notin\Omega_A$.
Since $\mathbb P(\Omega_A)=1$, it follows from the disintegration
formula that
\[
\psi_A(\omega,x)
=
E_{\mu_\omega}(1_{A_\omega}\mid\mathcal C_\omega)(x)
\]
for $\mu$-a.e. $(\omega,x)\in\mathcal E$.

Since $\beta$ is countable, we may choose a common $\mu$-full measure
set on which this identity holds for every $A\in\beta$. Therefore,
$$
\begin{aligned}
I_\mu(\beta\mid\mathcal C\vee\mathcal F_\E)(\omega,x)
&=
-\sum_{A\in\beta}
1_A(\omega,x)
\log E_\mu(1_A\mid\mathcal C\vee\mathcal F_\E)(\omega,x)\\
&=
-\sum_{A\in\beta}
1_{A_\omega}(x)
\log E_{\mu_\omega}
(1_{A_\omega}\mid\mathcal C_\omega)(x)\\
&=
I_{\mu_\omega}(\beta_\omega\mid\mathcal C_\omega)(x)
\end{aligned}
$$
for $\mu$-a.e. $(\omega,x)\in\E$. This proves
\eqref{eq fiber to global}.
\end{proof}

\begin{lemma}\label{lemma entropy computation}
  For $\mu\in \mathcal{I}_\P(\E)$ and any finite or countable measurable partition $\alpha$ of $\E$ satisfying $H_\mu(\alpha|\F_\E)<\infty$, one has
  \begin{equation*}
    h_\mu(F,\alpha)=H_\mu(\alpha|\alpha^-\vee\F_\E).
  \end{equation*}
\end{lemma}
\begin{proof}[Proof of Lemma \ref{lemma entropy computation}]
  By \cite[Theorem 4.2 (9)]{liuqiansmooth95}, $h_\mu^{(r)}(\Theta,\alpha)=h_{\mu}^{(r)}(\Theta^{-1},\alpha)$. Therefore, we have
  \begin{equation*}
    \begin{split}
      h_\mu(F,\alpha)  &= h_\mu^{(r)}(\Theta,\alpha)= h_{\mu}^{(r)}(\Theta^{-1},\alpha)=\lim_{n\to+\infty}\frac{1}{n} H_\mu(\alpha\vee \Theta\alpha\vee\cdots \vee \Theta^{n-1}\alpha|\F_\E)\\
         &=\lim_{n\to+\infty}\frac{1}{n}\left(H_\mu(\alpha|\F_\E)+H_\mu(\alpha|\Theta\alpha\vee\F_\E)+\cdots +H_\mu(\alpha|\vee_{i=1}^{n-1}\Theta^i\alpha\vee \F_\E)\right)\\
         &=\lim_{n\to+\infty} H_\mu(\alpha|\vee_{i=1}^{n-1}\Theta^i\alpha\vee \F_\E)=H_\mu(\alpha|\alpha^-\vee \F_\E).
    \end{split}
  \end{equation*}
This completes the proof.
\end{proof}

\begin{lemma}\label{lemma three of entropy}
  Let $\mu\in \ip$, and let $\alpha,\beta,\gamma$ be finite or countable measurable partitions of $\E$ satisfying $H_\mu(\alpha|\F_\E)<\infty$, $H_\mu(\beta|\F_\E)<\infty$, and $H_\mu(\gamma|\F_\E)<\infty$. Then
  \begin{enumerate}
    \item If $\beta\prec\alpha$, then $\lim\limits_{n\to+\infty}\frac{1}{n}H_\mu(\alpha_0^n|\beta^{-}\vee\F_\E)=H_\mu(\alpha|\alpha^{-}\vee \F_\E)$.
    \item If $\alpha\prec \beta$, then $\lim\limits_{n\to+\infty}\frac{1}{n}H_\mu(\alpha_0^n|\beta^{-}\vee\F_\E)=H_\mu(\alpha|\alpha^{-}\vee \F_\E).$
    \item If $\alpha\prec \beta$, then $\lim\limits_{n\to+\infty} H_\mu(\alpha|\beta^{-}\vee \Theta^n\gamma^{-}\vee \F_\E)=H_\mu(\alpha|\beta^{-}\vee \F_\E).$
  \end{enumerate}
\end{lemma}
\begin{proof}[Proof of Lemma \ref{lemma three of entropy}]
This lemma is a relative version of \cite[Lemma 18.2]{Glasner}.
  We first prove part (1). Assume that  $\beta\prec \alpha$. Then $$\Theta^n(\beta^-\vee \alpha_0^{n-1})\nearrow \alpha^-.$$
 Indeed, $$
\Theta^n\bigl(\beta^-\vee\alpha_0^{n-1}\bigr)
=
\left(\bigvee_{i=1}^{n}\Theta^i\alpha\right)
\vee
\left(\bigvee_{i=n+1}^{+\infty}\Theta^i\beta\right),
$$
and the assumption $\beta\prec\alpha$ implies that this is an
increasing sequence of partitions whose limit is $\alpha^-$.

 By the invariance of $\mu$, the equality $\Theta^{-1}\F_\E=\F_\E$, and the martingale convergence theorem,
  \begin{equation}\label{eq H convergence 1}
    H_\mu(\Theta^{-n}\alpha|\alpha_{0}^{n-1}\vee \beta^-\vee \F_\E)= H_\mu(\alpha|\Theta^n(\alpha_{0}^{n-1}\vee \beta^-)\vee \F_\E)\searrow H_\mu(\alpha|\alpha^-\vee \F_\E).
  \end{equation}
Therefore, we have
  \begin{equation*}
  \begin{split}
     &\lim\limits_{n\to+\infty}\frac{1}{n}H_\mu(\alpha_0^n|\beta^{-}\vee\F_\E)\\ =&  \lim\limits_{n\to+\infty}\frac{1}{n}(H_\mu(\alpha|\beta^{-}\vee \F_\E)+H_\mu(\Theta^{-1}\alpha|\alpha\vee\beta^{-}\vee \F_\E )+\cdots +H_\mu(\Theta^{-n}\alpha|\alpha_{0}^{n-1}\vee \beta^-\vee \F_\E))\\
     =&\lim\limits_{n\to+\infty}H_\mu(\Theta^{-n}\alpha|\alpha_{0}^{n-1}\vee \beta^-\vee \F_\E)\\
     =&H_\mu(\alpha|\alpha^-\vee \F_\E).
  \end{split}
  \end{equation*}
The first equality follows from the entropy cocycle equation, the second equality follows from the Cesàro convergence theorem, and the third equality follows from \eqref{eq H convergence 1}. This proves part (1).

Next, we prove part (2).  We assume $\alpha\prec \beta$. Then $ \alpha^{-} \prec \beta^{-}$, and hence
\begin{equation*}
\limsup\limits_{n\to+\infty}\frac{1}{n}H_\mu(\alpha_0^n|\beta^{-}\vee\F_\E)\leq \lim\limits_{n\to+\infty}\frac{1}{n}H_\mu(\alpha_0^n|\alpha^{-}\vee\F_\E)=H_\mu(\alpha|\alpha^-\vee\F_\E),
\end{equation*}
where the last equality follows from part (1). Notice that
\begin{equation*}
  \begin{split}
  \frac{1}{n}H_\mu(\alpha_0^n|\beta^{-}\vee\F_\E)  =&\frac{1}{n}H_\mu(\beta_0^n|\beta^{-}\vee\F_\E)-\frac{1}{n}H_\mu(\beta_0^n|\alpha_0^n\vee \beta^{-}\vee\F_\E)\\
  \geq & \frac{1}{n}H_\mu(\beta_0^n|\beta^{-}\vee\F_\E)-\frac{1}{n}H_\mu(\beta_0^n|\alpha_0^n\vee \alpha^{-}\vee\F_\E)\\
  =&\frac{1}{n}H_\mu(\beta_0^n\mid\beta^-\vee\F_\E) -  \frac{1}{n}H_\mu(\beta_0^n\mid\alpha^-\vee\F_\E)
+
 \frac{1}{n}H_\mu(\alpha_0^n\mid\alpha^-\vee\F_\E),
  \end{split}
\end{equation*}
 where the first and third equalities follow from the fact that $\alpha_0^n\prec\beta_0^n$ and from the entropy cocycle equation, and the inequality follows because $\beta^-$ is finer than $\alpha^-$.
By part (1),
$$
\lim_{n\to+\infty}
\frac{1}{n}
H_\mu(\beta_0^n\mid\beta^-\vee\F_\E)
=
H_\mu(\beta\mid\beta^-\vee\F_\E),
$$
$$
\lim_{n\to+\infty}
\frac{1}{n}
H_\mu(\beta_0^n\mid\alpha^-\vee\F_\E)
=
H_\mu(\beta\mid\beta^-\vee\F_\E),
$$
and
$$
\lim_{n\to+\infty}
\frac{1}{n}
H_\mu(\alpha_0^n\mid\alpha^-\vee\F_\E)
=
H_\mu(\alpha\mid\alpha^-\vee\F_\E).
$$
Consequently,
$$
\liminf_{n\to+\infty}
\frac{1}{n}
H_\mu(\alpha_0^n\mid\beta^-\vee\F_\E)
\geq
H_\mu(\alpha\mid\alpha^-\vee\F_\E).
$$
Combining this with the preceding $\limsup$ estimate, we conclude that
$$
\lim_{n\to+\infty}
\frac{1}{n}
H_\mu(\alpha_0^n\mid\beta^-\vee\F_\E)
=
H_\mu(\alpha\mid\alpha^-\vee\F_\E).
$$
This proves part (2).

Finally, we prove part (3). Assume that $\alpha\prec\beta$  and
$H_\mu(\gamma\mid\F_\E)<+\infty$. Since
$$
H_\mu(\beta\vee\gamma\mid\F_\E)
\leq
H_\mu(\beta\mid\F_\E)
+
H_\mu(\gamma\mid\F_\E)
<+\infty,
$$
part (2) applies to $\beta\prec\beta\vee\gamma$.

Since $\beta^-\vee\Theta^n\gamma^-\vee\F_\E$ is finer than
$\beta^-\vee\F_\E$, we have
$$
\limsup_{n\to+\infty}
H_\mu(\alpha\mid\beta^-\vee\Theta^n\gamma^-\vee\F_\E)
\leq
H_\mu(\alpha\mid\beta^-\vee\F_\E).
$$

By part (2), the entropy cocycle equation, and the
$\Theta$-invariance of $\mu$, one has
{\small
\begin{equation}\label{eq H convergence 2}
\begin{split}
H_\mu(\beta\mid\beta^-\vee\F_\E)
&=
\lim_{n\to+\infty}
\frac{1}{n}
H_\mu(\beta_0^n\mid\beta^-\vee\gamma^-\vee\F_\E)\\
&=
\lim_{n\to+\infty}
\frac{1}{n}
\sum_{k=0}^{n}
H_\mu\bigl(
\beta\mid\beta^-\vee\Theta^k\gamma^-\vee\F_\E
\bigr)\\
&=
\lim_{n\to+\infty}
H_\mu\bigl(
\beta\mid\beta^-\vee\Theta^n\gamma^-\vee\F_\E
\bigr).
\end{split}
\end{equation}}
Indeed, since
$
\Theta^n\gamma^-
=
\bigvee_{j=n+1}^{+\infty}\Theta^j\gamma,
$
the sequence
$
H_\mu\bigl(
\beta\mid\beta^-\vee\Theta^n\gamma^-\vee\F_\E
\bigr)
$
is increasing and bounded above by
$H_\mu(\beta\mid\beta^-\vee\F_\E)$. Hence it converges, and the last
equality follows from the Ces\`aro convergence theorem. Here, the fact
that the sum contains $n+1$ terms causes no problem, since
$(n+1)/n\to1$.

Again, by the entropy cocycle equation,
\begin{equation*}
\begin{split}
&H_\mu(\alpha\mid\beta^-\vee\Theta^n\gamma^-\vee\F_\E)\\
&=
H_\mu(\beta\mid\beta^-\vee\Theta^n\gamma^-\vee\F_\E)
-
H_\mu\bigl(
\beta\mid
\alpha\vee\beta^-\vee\Theta^n\gamma^-\vee\F_\E
\bigr)\\
&\geq
H_\mu(\beta\mid\beta^-\vee\Theta^n\gamma^-\vee\F_\E)
-
H_\mu(\beta\mid\alpha\vee\beta^-\vee\F_\E).
\end{split}
\end{equation*}
Therefore, by combining this with \eqref{eq H convergence 2}, we obtain
\begin{equation*}
\begin{split}
&\liminf_{n\to+\infty}
H_\mu(\alpha\mid\beta^-\vee\Theta^n\gamma^-\vee\F_\E)\\
&\geq
H_\mu(\beta\mid\beta^-\vee\F_\E)
-
H_\mu(\beta\mid\alpha\vee\beta^-\vee\F_\E)\\
&=
H_\mu(\alpha\mid\beta^-\vee\F_\E),
\end{split}
\end{equation*}
where the last equality follows from the entropy cocycle equation.

Combining this with the preceding $\limsup$ estimate, we conclude that
$$
\lim_{n\to+\infty}
H_\mu(\alpha\mid\beta^-\vee\Theta^n\gamma^-\vee\F_\E)
=
H_\mu(\alpha\mid\beta^-\vee\F_\E).
$$
This proves (3) and completes the proof of Lemma
\ref{lemma three of entropy}.
\end{proof}

\subsection{A local version of the Shannon--McMillan--Breiman theorem}
Before stating and proving a local version of the Shannon--McMillan--Breiman theorem for RDSs, we need the following lemma of Breiman, which can be found in \cite[Lemma 14.34]{Glasner}.
\begin{lemma}\label{lemma Breiman}
  Let $(Y,\mathcal{Y},\nu,T)$ be a measure-preserving dynamical system, $g_n,g\in L^1(\nu)$ with $g_n\to g$ a.e. and in $L^1(\nu)$. If $\int\sup_n|g_n|<\infty$, then
  \begin{equation*}
    \frac{1}{N}\sum_{n=1}^{N}g_n(T^ny)\to E_\nu(g|\mathcal{I}_T)\mbox{ a.e. and in $L^1(\nu)$},
  \end{equation*}
where $\mathcal{I}_T$ is the $\sigma$-algebra of $T$-invariant sets.
\end{lemma}
The following lemma is a local version of the Shannon--McMillan--Breiman theorem for RDSs and may be viewed as a relative counterpart of \cite[Theorem 14.35]{Glasner}; see also \cite[Proposition 3.1]{Fengscichina2022} and \cite[Theorem 2.2.5]{bogenschiitz1993equilibrium}.
\begin{lemma}\label{lemma random smb theorem}
Assume $\mu\in \mathcal{I}_\P(\E)$. Let $\xi$ be a countably generated measurable partition of $\E$ with $\Theta\xi\prec \xi$. Let $\alpha$ be a finite measurable partition satisfying $\alpha^-\prec \xi$. Define the measurable partition $\gamma_n:=\alpha^-\vee\Theta^n\xi$ for all $n\geq 0$ and a $\sigma$-algebra $\mathcal{H}:=\bigcap\limits_{n=0}^\infty\left(\sigma(\gamma_n)\vee \F_\E\right)$. Then we have
\begin{equation*}
  \lim_{N\to+\infty}\frac{1}{N}I_{\mu_\omega}\left((\alpha_0^{N-1})_\omega|\xi_\omega\right)(x)=E_\mu(f|\mathcal{I}_\Theta)(\omega,x)\mbox{ $\mu$-a.e. and in $L^1(\mu)$,}
\end{equation*}
where
$
  f(\omega,x):=I_{\mu}(\alpha|\mathcal{H})(\omega,x),
$ and $\mathcal{I}_\Theta$ is the $\sigma$-algebra of $\Theta$-invariant sets.
\end{lemma}
\begin{proof}[Proof of Lemma \ref{lemma random smb theorem}]

  We write
   \begin{equation*}
     f_0(\omega,x):=I_\mu(\alpha|\xi\vee\F_\E)(\omega,x)\overset{\eqref{eq fiber to global}}{=}I_{\mu_\omega}(\alpha_\omega|\xi_\omega)(x),\mbox{ $\mu$-a.e.}
   \end{equation*}
and, for $n\geq 1$,
   \begin{equation*}
     f_n(\omega,x):= I_\mu(\alpha|\vee_{i=1}^n\Theta^i\alpha\vee\Theta^n\xi\vee \F_\E)(\omega,x) \overset{\eqref{eq fiber to global}}{=} I_{\mu_\omega}(\alpha_\omega| \vee_{i=1}^n(\Theta^i\alpha)_\omega\vee(\Theta^n\xi)_\omega)(x),\ \mbox{$\mu$-a.e.}
   \end{equation*}
Then we have
  \begin{equation*}
  \begin{split}
     I_{\mu_\omega}((\alpha_0^{N-1})_\omega|\xi_\omega) (x) =&I_{\mu_\omega}(\alpha_\omega|\xi_\omega)(x)+\sum_{n=1}^{N-1}I_{\mu_\omega}((\Theta^{-n}\alpha)_\omega|(\alpha_0^{n-1})_\omega\vee\xi_\omega)(x) \\
       =&I_{\mu_\omega}(\alpha_\omega|\xi_\omega)(x)+\sum_{n=1}^{N-1}I_{\mu_{\theta^n\omega}}(\alpha_{\theta^n\omega}|\vee_{i=1}^n(\Theta^i\alpha)_{\theta^n\omega}\vee(\Theta^n\xi)_{\theta^n\omega})(F_\omega^nx)\\
       =&\sum_{n=0}^{N-1}f_n\circ \Theta^n(\omega,x).
  \end{split}
  \end{equation*}
Since $\Theta\xi\prec \xi$, we have $\sigma(\gamma_n)\vee \F_\E\searrow\mathcal{H}$. Due to the assumption $\alpha^-\prec \xi$, we have $\vee_{i=1}^n\Theta^i\alpha\vee \Theta^n\xi\vee \F_\E= \vee_{i=1}^\infty\Theta^i\alpha\vee \Theta^n\xi \vee \F_\E\searrow\mathcal{H}$. 
Since $\alpha$ is finite, $H_\mu(\alpha)<\infty$. Hence,
applying Lemma~\ref{lemma four useful lemma}(4)--(5) to this
decreasing sequence of $\sigma$-algebras, we obtain
  \begin{equation*}
    \lim_{n\to+\infty}f_n(\omega,x)=I_\mu(\alpha|\mathcal{H})(\omega,x)
    \ \mu\mbox{-a.e. and in $L^1(\mu)$},
  \end{equation*}
and
\begin{equation*}
\int_{\mathcal E}\sup_{n\geq0}f_n(\omega,x)\,
d\mu(\omega,x)
\leq H_\mu(\alpha)+1<\infty.
\end{equation*}
Therefore, by Lemma \ref{lemma Breiman}, we have
  \begin{equation*}
       \lim_{N\to+\infty}\frac{1}{N}I_{\mu_\omega}\left((\alpha_0^{N-1})_\omega|\xi_\omega\right)(x)  =  \lim_{N\to+\infty}\frac{1}{N}\sum_{n=0}^{N-1}f_n\circ \Theta^n(\omega,x)=E_\mu(f|\mathcal{I}_\Theta)(\omega,x)
  \end{equation*}
for $\mu$-a.e. $(\omega,x)\in\E$ and in $L^1(\mu)$.
\end{proof}

\begin{corollary}\label{corollary random smb}
  Assume the hypotheses of Lemma \ref{lemma random smb theorem}, and assume further that $\mu\in \mathcal{I}_\P^e(\E)$. Let $(\mu_\omega)_x$ be the disintegration of $\mu$ over $\xi\vee\F_\E$ provided by Corollary \ref{corollary measure disintegration}. Then there exists a measurable set $Y\subset\E$ with $\mu(Y)=1$ such that, for $\P$-a.e. $\omega\in\Omega$ and every $x\in Y_\omega$, if
  \begin{equation*}
    \Xi_\omega(x):=\xi_\omega(x)\cap Y_\omega,
  \end{equation*}
then $(\mu_\omega)_x( \Xi_\omega(x))=1$ and
  \begin{equation*}
    \lim_{N\to+\infty}\frac{-\log(\mu_\omega)_x\left((\alpha_0^{N-1})_\omega(y)\right)}{N}=H_\mu(\alpha|\mathcal{H}),\ \forall y\in \Xi_\omega(x).
  \end{equation*}
\end{corollary}
\begin{proof}[Proof of Corollary \ref{corollary random smb}]

  According to Corollary \ref{corollary measure disintegration},
  there exists a $\mu$-full-measure set $\E^\prime\subset\E$ satisfying $(\mu_\omega)_x(\xi_\omega(x)\cap \E^\prime_\omega)=1$ for all $(\omega,x)\in \E^\prime$. Moreover, for any $x,y\in\E_\omega^\prime$, $\xi_\omega(x)=\xi_\omega(y)$ implies that $(\mu_\omega)_x=(\mu_\omega)_y$.
  By the definition of the information function, we have
  \begin{equation*}
    I_{\mu_\omega}((\alpha_0^{N-1})_\omega|\xi_\omega)(x)=-\sum_{A\in (\alpha_0^{N-1})_\omega}1_A(x)\log E_{\mu_\omega}(1_A|\xi_\omega)(x)=-\log(\mu_\omega)_x((\alpha_0^{N-1})_\omega(x)),\mbox{ $\mu$-a.e.},
  \end{equation*}
where the second equality follows from Lemma \ref{lemma measure disintegration} (1). Since $\mu$ is $\Theta$-ergodic, the $\Theta$-invariant function $E_\mu(f|\mathcal{I}_{\Theta})$ defined in Lemma \ref{lemma random smb theorem} is a constant function, and moreover,
  \begin{equation*}
    E_\mu(f|\mathcal{I}_{\Theta})(\omega,x)=
    \int_\E I_\mu(\alpha|\mathcal{H}) d\mu= H_\mu(\alpha|\mathcal{H}),\ \mbox{$\mu$-a.e.}
  \end{equation*}

By Lemma \ref{lemma random smb theorem} and the preceding identity, there is a measurable set
$Y_0\subset\mathcal E'$ with $\mu(Y_0)=1$ such that
$$
\lim_{N\to\infty}
\frac{-\log(\mu_\omega)_y
((\alpha_0^{N-1})_\omega(y))}{N}
=H_\mu(\alpha | \mathcal H)
$$
for every $(\omega,y)\in Y_0$. Choose the conditional kernel from
Corollary \ref{corollary measure disintegration} on the whole space, changing it on a $\mu$-null set
if necessary, and put
$$
q(\omega,y):=(\mu_\omega)_y((Y_0)_\omega),\qquad
G:=\{(\omega,y)\in\mathcal E':q(\omega,y)=1\}.
$$
The function $q$ is a version of
$ E_\mu(\mathbf 1_{Y_0}\mid
\xi\vee\mathcal F_{\mathcal E})$. Hence
$G\in\sigma(\xi\vee\mathcal F_{\mathcal E})$ and $\mu(G)=1$.
Set $Y:=Y_0\cap G$. If $(\omega,y)\in Y$, then the atom of
$\sigma(\xi\vee\mathcal F_{\mathcal E})$ containing
$(\omega,y)$ is contained in $G$. Since
$(\mu_\omega)_y$ is concentrated on this atom,
$(\mu_\omega)_y(G_\omega)=1$. Together with $q(\omega, y)=1$,
this gives
$$
(\mu_\omega)_y(\xi_\omega(y)\cap Y_\omega)=1.
$$
Finally, if $x,y\in Y_\omega$ and
$\xi_\omega(x)=\xi_\omega(y)$, then
$(\mu_\omega)_x=(\mu_\omega)_y$. Applying the convergence on
$Y_0$ at $y$ therefore gives
$$
\lim_{N\to\infty}
\frac{-\log(\mu_\omega)_x
((\alpha_0^{N-1})_\omega(y))}{N}
=H_\mu(\alpha | \mathcal H),
$$
 for every $x\in Y_\omega$ and every
$y\in\xi_\omega(x)\cap Y_\omega$.
\end{proof}

\subsection{Local entropy}
Let $(\E,F)$ be a two-sided continuous bundle RDS. For $n\in\mathbb{N}$, $\omega\in\Omega$, $x\in \E_\omega$, and $\epsilon>0$, we define the $n$-step fiber Bowen ball by
\begin{equation*}
  B_\omega(x,\epsilon,n):=\{y\in \E_\omega:\ \max_{0\leq k\leq n-1}d(F_\omega^kx,F_\omega^ky)<\epsilon\}.
\end{equation*}
\begin{definition}\label{def local entropy}
  Given a Borel probability measure $\upsilon$ concentrated on $\E_\omega$, the fiber lower and upper local entropies of $F$ are defined, respectively, by
  \begin{equation*}
    \begin{split}
       \underbar{h}_\upsilon(F,\omega,x) & =\lim_{\epsilon\to 0}\liminf_{n\to+\infty}\frac{-\log\upsilon(B_\omega(x,\epsilon,n))}{n}, \\
         \overline{h}_\upsilon(F,\omega,x) & =\lim_{\epsilon\to 0}\limsup_{n\to+\infty}\frac{-\log\upsilon(B_\omega(x,\epsilon,n))}{n}.
    \end{split}
  \end{equation*}
\end{definition}
When $\mu\in \mathcal{I}_\P^e(\E)$, Zhu proved in \cite[Theorem 2.1]{Zhu2009} that
\begin{equation}\label{eq brin-katok}
  h_\mu(F)=\underline{h}_{\mu_\omega}(F,\omega,x)=\overline{h}_{\mu_\omega}(F,\omega,x)\mbox{ for $\mu$-a.e. }(\omega,x)\in\E,
\end{equation}
where $\mu_\omega$ is the disintegration of $\mu$ over $\mathcal{F}_\E$.

\section{Construction of the partition subordinate to the fiber local unstable sets}\label{section proof thm 2}
\subsection{Proof of Theorem \ref{thm local entropy}, Part (1)}
Let $\delta>0$ be as in the statement of Theorem \ref{thm local entropy}. Let $\mu\in \mathcal{I}_\P^e(\E)$ with $h_\mu(F)>0$.
Let $\{\mathcal{Q}_i\}_{i\in\mathbb{Z}_+}$ be a sequence of finite Borel measurable partitions of $X$ satisfying
\begin{enumerate}
  \item[(i)] $\mathcal{Q}_{i+1}$ is finer than $\mathcal{Q}_i$ for $i\geq 1$, diam$\mathcal{Q}_1\leq \delta$, and $\lim_{i\to+\infty}$diam$(\mathcal{Q}_i)=0$,
  \item[(ii)]  $\int_{\Omega} \mu_{\omega}(\partial \mathcal{Q}_i\cap \E_{\omega}) d\mathbb{P}(\omega)=0$ for every $i\in \mathbb{Z}_{+}$.
\end{enumerate}
We define the finite measurable partition $\hat{\mathcal{Q}}_i:=\{(\Omega\times A)\cap \E:\ A\in\mathcal{Q}_i\}$ to be the partition induced by $\mathcal{Q}_i$ on $\E$, and define
\begin{equation}\label{eq def beta i}
  \beta_i:=\hat{\mathcal{Q}_i}\vee \Theta^{-1}\hat{\mathcal{Q}}_1.
\end{equation}
Then $\beta_i$ is a finite measurable partition of $\E$ satisfying
\begin{enumerate}
  \item $\beta_1\prec \beta_2\prec \cdots$, and $\lim_{i\to+\infty}\sup_{\omega\in\Omega}$diam$(\beta_i)_\omega=0$,
  \item $\sup_{\omega\in\Omega}$diam$(\beta_1)_\omega\leq$diam$\mathcal{Q}_1\leq \delta$ and $\sup_{\omega\in\Omega}\max_{A\in \beta_1}{\rm diam}(F_\omega(A))\leq {\rm diam}\mathcal{Q}_1\leq \delta$,
  \item for every $i\in\mathbb{Z}_+$, $\int_\Omega\mu_\omega(\partial (\beta_i)_\omega)d\P(\omega)=0$.
\end{enumerate}
The first property of $\beta_i$ guarantees that the countably generated measurable partition $\vee_{i=1}^\infty\beta_i$ is a fiberwise topological generator of $F$, i.e., $x=y$ whenever $\Theta^n(\omega,x)$ and $\Theta^n(\omega,y)$ belong to the same element of $\vee_{i=1}^\infty\beta_i$ for all $n\in\mathbb{Z}$. As a consequence of \cite[Theorem 2.3.3]{bogenschiitz1993equilibrium} (see also \cite[Theorem 1.1.2]{KL06}), one has
\begin{equation}\label{eq beta approach entropy}
  \sup_{i\geq 1}h_\mu(F,\beta_i)=h_\mu(F,\vee_{i=1}^\infty\beta_i)=h_\mu(F).
\end{equation}
Set $k_1=0$. We can find inductively a strictly increasing sequence of integers $k_i$ for $i=1,2,...$ such that, with 
 \begin{eqnarray}\label{defn-4.1.1}
  \alpha_q:=\vee_{i=1}^q\Theta^{k_i}\beta_i,
\end{eqnarray}
the following inequality holds for every $q\geq 2$ and $p=1,2,\ldots,q-1$:
\begin{equation}\label{eq h-h<1/p2q-p}
  H_\mu(\alpha_p|(\alpha_{q-1})^-\vee \F_\E)-H_\mu(\alpha_p|(\alpha_{q})^-\vee \F_\E)<\frac{1}{p\cdot 2^{q-p}}.
\end{equation}
The integer $k_q$ is chosen inductively as follows. Once $k_p$ has been chosen for $p\geq 1$, applying Lemma \ref{lemma three of entropy} (3) to $\alpha=\alpha_p$, $\beta=\alpha_{q-1}$ and $\gamma=\beta_q$ for each   $q> p\geq 1$, we can find a $k_q$ as desired.

We define
\begin{equation}\label{defn-4.1.2}
  \mathcal{P}:=\bigvee_{p=1}^\infty\alpha_p,\mbox{ and }\xi:=\mathcal{P}^-=\bigvee_{n=1}^\infty\Theta^n\mathcal{P}.
\end{equation}
Directly, we have
\begin{equation*}
  \Theta\xi\prec \xi,\ (\alpha_p)^-\prec \xi,\mbox{ and }\alpha_p\nearrow\mathcal{P}.
\end{equation*}

By construction, we have $\overline{\xi_\omega(x)}\subset \overline{(F_{\theta^{-1}\omega}(\beta_1)_{\theta^{-1}\omega})(x)}$, and
\begin{equation*}
  F_\omega^{-n}\overline{\xi_\omega(x)}=\overline{F_\omega^{-n}\xi_\omega(x)}=\overline{(\Theta^{-n}\xi)_{\theta^{-n}\omega}(F_\omega^{-n}x)}\subset \overline{(\beta_1)_{\theta^{-n}\omega}(F_\omega^{-n}x)},\ \forall n\geq 1.
\end{equation*}
Thus, we have
\begin{equation}\label{eqn-4.1.1}
  \mbox{diam} (F_\omega^{-n}\overline{\xi_\omega(x)})\leq \max\{\mbox{diam} (F_{\theta^{-1}\omega}(\beta_1)_{\theta^{-1}\omega},\mbox{diam}(\beta_1)_{\theta^{-n}\omega}\}\leq \delta,\ \forall n\geq 0.
\end{equation}
For any $j\geq 1$ and $i\geq 1$, we have
\begin{equation*}\label{eqn-4.1.2}
  F_\omega^{-(k_j+i)}\overline{\xi_\omega(x)}=\overline{F_\omega^{-(k_j+i)}\xi_\omega(x)}=\overline{(\Theta^{-(k_j+i)}\xi)_{\theta^{-(k_j+i)}\omega}(F_\omega^{-(k_j+i)}x)}
  \subset \overline{(\beta_j)_{\theta^{-(k_j+i)}\omega}(F_\omega^{-(k_j+i)}x)}.
\end{equation*}
Moreover,
\begin{equation*}
\operatorname{diam}(F_\omega^{-(k_j+i)}\overline{\xi_\omega(x)})
\leq \sup_{\omega\in\Omega} \operatorname{diam}(\beta_j)_\omega\to 0
\quad\text{as }j\to +\infty.
\end{equation*}
Combining this estimate with the preceding inclusion and \eqref{eqn-4.1.1}, we obtain $\overline{\xi_\omega(x)}\subset W_\omega^u(x,\delta)$. This completes the proof of Part (1) of Theorem \ref{thm local entropy}.

\subsection{Proof of Theorem \ref{thm local entropy}, Part (2)}
Let $\alpha_p$ be the partition defined in \eqref{defn-4.1.1}, and let \(\xi\) be the countably generated measurable partition defined in \eqref{defn-4.1.2}.
We define
\begin{equation}\label{defn-4.2.1}
   \mathcal{H}_p:=\bigcap_{n=0}^\infty\left(\sigma\{(\alpha_p)^-\vee \Theta^n\xi\}\vee \F_\E\right).
\end{equation}
We then have the following lemma.
\begin{lemma}\label{lemma entropy approach by alpha and H}
We have the following equalities
\begin{equation}\label{eq entropy approach by alpha and H}
  h_\mu(F)=\sup_{p\geq 1}h_\mu(F,\alpha_p),\mbox{ and }h_\mu(F)=\sup_{p\geq 1}H_\mu(\alpha_p|\mathcal{H}_p).
\end{equation}
\end{lemma}
\begin{proof}[Proof of Lemma \ref{lemma entropy approach by alpha and H}]
  It is clear that $h_\mu(F)\geq \sup_{p\geq 1}h_\mu(F,\alpha_p)$. Note that
\begin{equation*}
  h_\mu(F,\alpha_p)=h_\mu(F,\vee_{i=1}^p\Theta^{k_i}\beta_i)\geq \max_{1\leq i\leq p}h_\mu(F,\Theta^{k_i}\beta_i)=\max_{1\leq i\leq p}h_\mu(F,\beta_i).
\end{equation*}
Taking the supremum over $p\geq 1$ and using $\eqref{eq beta approach entropy}$, we obtain $\sup_{p\geq 1}h_\mu(F,\alpha_p)\geq h_\mu(F)$. Thus, we have $h_\mu(F)=\sup_{p\geq 1}h_\mu(F,\alpha_p)$.

Given $p,m\in\mathbb{N}$, summing $\eqref{eq h-h<1/p2q-p}$ for $q=p+1,\ldots,p+m$ yields
\begin{equation*}
  H_\mu(\alpha_p|(\alpha_p)^-\vee\F_\E)-H_\mu(\alpha_p|(\alpha_{p+m})^-\vee\F_\E)<\frac{1}{p}.
\end{equation*}
Letting $m\to+\infty$, one has
$
  H_\mu(\alpha_p|(\alpha_p)^-\vee\F_\E)-H_\mu(\alpha_p|\xi\vee\F_\E)<\frac{1}{p}.
$ By Lemma \ref{lemma entropy computation}, $h_\mu(F,\alpha_p)-H_\mu(\alpha_p|\xi\vee\F_\E)<\frac{1}{p}$. Since $(\alpha_p)^-\vee\F_\E\prec (\alpha_p)^-\vee \Theta^n\xi\vee \F_\E\prec \xi\vee \F_\E$ for all $n\in\mathbb{Z}$, and since $\mathcal{H}_p:=\bigcap_{n=0}^\infty(\sigma\{(\alpha_p)^-\vee \Theta^n\xi\}\vee \F_\E)$, we have
\begin{equation*}
  h_\mu(F,\alpha_p)-\frac{1}{p}\leq H_\mu(\alpha_p|\xi\vee\F_\E)\leq H_\mu(\alpha_p|\mathcal{H}_p)\leq H_\mu(\alpha_p|(\alpha_p)^-\vee\F_\E)=h_\mu(F,\alpha_p).
\end{equation*}
Taking the supremum over $p\geq 1$, we obtain
$
  h_\mu(F)=\sup_{p\geq 1}h_\mu(F,\alpha_p)=\sup_{p\geq 1}H_\mu(\alpha_p|\mathcal{H}_p).
$
\end{proof}

Let $\mu_\omega$ denote the disintegration of $\mu$ over $\F_\E$, and let $(\mu_\omega)_x$ be the disintegration of $\mu_\omega$ over $\xi_\omega$ provided by Corollary \ref{corollary measure disintegration}. We apply Lemma \ref{lemma random smb theorem} and Corollary \ref{corollary random smb} with $\alpha=\alpha_p$ and $\mathcal{H}=\mathcal{H}_p$, where $\alpha_p$ is defined in \eqref{defn-4.1.1} and $\mathcal{H}_p$ is defined in \eqref{defn-4.2.1}.
Note that $\alpha=\alpha_p$ is a finite measurable partition. Then, for $\mu$-a.e. $(\omega,x)\in \E$ and $(\mu_\omega)_x$-a.e. $y\in \E_\omega$,
\begin{equation}\label{eq proof smb in part 2}
  \lim_{N\to+\infty}\frac{-\log(\mu_\omega)_x(\vee_{k=0}^{N-1}(\Theta^{-k}\alpha_p)_\omega(y))}{N}=H_\mu(\alpha_p|\mathcal{H}_p).
\end{equation}
The following lemma can be found in \cite[Lemma 3.4]{Fengscichina2022}.
\begin{lemma}\label{lemma huang}
  Let $\upsilon$ be a Borel probability measure on a compact metric space $X$. Let $\eta$ be a finite partition of $X$, and let $\{\gamma_i\}_{i\in \mathbb{Z}_+}$  be a refining sequence of finite partitions of $X$, that is $\gamma_1\prec\cdots \prec \gamma_n\prec\cdots.$
   If there exists a constant $c$ such that
  \begin{equation*}
    \lim_{n\to+\infty}\frac{-\log\nu(\eta\vee\gamma_n(x))}{n}=c,\ \mbox{for $\nu$-a.e. $x\in X$},
  \end{equation*}
then the following conclusion holds:
  \begin{equation*}
    \lim_{n\to+\infty}\frac{-\log\nu(\gamma_n(x))}{n}=c,\ \mbox{for $\nu$-a.e. $x\in X$}.
  \end{equation*}
\end{lemma}
Taking $\eta=\vee_{i=0}^{k_p-1}(\Theta^{-i}\alpha_p)_\omega$ and $\gamma_N=\vee_{i=0}^{N-1}(\Theta^{-(k_p+i)}\alpha_p)_\omega$ in Lemma \ref{lemma huang}, we obtain, for $\mu$-a.e. $(\omega,x)\in \E$ and $(\mu_\omega)_x$-a.e. $y\in \E_\omega$,
\begin{equation*}
  \lim_{N\to+\infty}\frac{-\log(\mu_\omega)_x(\vee_{i=0}^{N-1}(\Theta^{-(k_p+i)}\alpha_p)_\omega(y))}{N}=H_\mu(\alpha_p|\mathcal{H}_p).
\end{equation*}
Now, for any $\epsilon>0$, since $\lim_{j\to+\infty}\sup_{\omega\in\Omega}\operatorname{diam}(\beta_j)_\omega=0$, we can find $N(\epsilon)\in\mathbb{N}$ such that $\sup_{\omega\in\Omega}\operatorname{diam}(\beta_j)_\omega\leq \epsilon$ for $j\geq N(\epsilon)$. Note that $\beta_p\prec \Theta^{-k_p}\alpha_p$; hence, when $p\geq N(\epsilon)$ and $n\geq 1$, we have
\begin{equation*}
  \left(\vee_{i=0}^{n-1}\Theta^{-(k_p+i)}\alpha_p\right)_\omega(y)\subset B_\omega(y,\epsilon,n).
\end{equation*}
Consequently, for $\mu$-a.e. $(\omega,x)\in \E$ and $(\mu_\omega)_x$-a.e. $y\in \E_\omega$, we have
\begin{equation*}
  \begin{split}
     \overline{h}_{(\mu_\omega)_x}(F,\omega,y) &=\lim_{\epsilon\to 0} \limsup_{n\to+\infty}\frac{-\log(\mu_\omega)_x(B_\omega(y,\epsilon,n))}{n}\\
       & \leq \sup_{p\geq 1}\limsup_{n\to+\infty}\frac{-\log (\mu_\omega)_x\left(\vee_{i=0}^{n-1}\Theta^{-(k_p+i)}\alpha_p\right)_\omega(y)}{n}\\
       &=\sup_{p\geq 1}H_\mu(\alpha_p|\mathcal{H}_p)=h_\mu(F),
  \end{split}
\end{equation*}
where the last equality follows from Lemma \ref{lemma entropy approach by alpha and H}, and we may take a countable sequence $\epsilon_k\to 0$ by taking a countable intersection
of the corresponding full-measure sets.

To finish the proof, it suffices to show that, for each fixed $p$, $\underbar{h}_{(\mu_\omega)_x}(F,\omega,y)\geq H_\mu(\alpha_p|\mathcal{H}_p)$ for $\mu$-a.e. $(\omega,x)\in\E$ and $(\mu_\omega)_x$-a.e. $y\in\E_\omega$. To simplify the notation, we denote $\alpha_p$ and $\mathcal{H}_p$ by $\alpha$ and $\H$, respectively, and set $c:=H_\mu(\alpha_p|\mathcal{H}_p)$.
\begin{lemma}\label{lemma lower bound for local entropy}
	There exists a sufficiently small constant $\epsilon_0>0$, independent of $\mu$ and $\omega\in\Omega$, such that, for every $\epsilon\in(0,\epsilon_0)$, there exist $\tau>0$ and a measurable set $\L\subset \E$ satisfying $\mu_\omega(\L_\omega)>1-\epsilon^{1/4}$ for $\P$-a.e. $\omega\in\Omega$ and having the following property. For any $x\in \L_\omega$, there exists a measurable subset $D(\omega,x)\subset \E_\omega$ satisfying $(\mu_\omega)_x(D(\omega,x))>1-4\epsilon^{1/4}$ and
	\begin{equation}\label{eq lower bound for local entropy}
		\liminf_{n\to+\infty}\frac{-\log(\mu_\omega)_x(B_\omega(y,\tau,n))}{n}\geq c-3(\Delta+\epsilon),\ \forall y\in D(\omega,x),
	\end{equation}
where
	\begin{equation}\label{eq delta}
	\Delta=	2\sqrt{\epsilon}\log(\#\alpha-1)-2\sqrt{\epsilon}\log 2\sqrt{\epsilon}-(1-2\sqrt{\epsilon})\log(1-2\sqrt{\epsilon}).
	\end{equation}
\end{lemma}

We postpone the proof of Lemma \ref{lemma lower bound for local entropy} and first finish the proof of Theorem \ref{thm local entropy}. We choose a sequence $\epsilon_k\in(0,\epsilon_0)$ decreasing to 0. Then there exist $\tau_k$, $\L_k$, and $D_k(\omega,x)$ for $x\in (\L_k)_\omega$ satisfying the conclusion of Lemma \ref{lemma lower bound for local entropy}. Replacing $\tau_k$ by $\min\{\tau_k,1/k\}$ if necessary, we may assume that $\tau_k<1/k$.
We consider $\overline{\L}=
\limsup_{n\to+\infty}\L_n=\cap_{N=1}^{+\infty}\cup_{k=N}\L_k$. Then $\mu(\overline{\L})=1$, and for \(\mu\)-a.e. \((\omega,x) \in \overline{\L}\), the point \(x\) belongs to \((\L_k)_\omega\) for infinitely many $k$. We denote the corresponding indices by \(k_1 < k_2 < \cdots\).
We consider $\overline{D}(\omega,x)=\limsup_{i\to+\infty} D_{k_i}(\omega,x)$. Then $(\mu_\omega)_x(\overline{D}(\omega,x))=1$ and
\begin{equation*}
	\underline{h}_{(\mu_\omega)_x}(F,\omega,y)\geq \lim_{i\to+\infty}	\liminf_{n\to+\infty}\frac{-\log(\mu_\omega)_x(B_\omega(y,\tau_{k_i},n))}{n}\geq c=H_\mu(\alpha_p|\mathcal{H}_p),\ \forall y\in \overline{D}(\omega,x),
\end{equation*}
Taking the intersection of the corresponding full-measure sets over $p\in\mathbb N$, we obtain a common full-measure set such that, for $\mu$-a.e. $(\omega,x)$,
\begin{equation*}
  \underline h_{(\mu_\omega)_x}(F,\omega,y)
\ge
\sup_{p\ge1}H_\mu(\alpha_p|\mathcal{H}_p)
\overset{\eqref{eq entropy approach by alpha and H}}{=}
h_\mu(F),\mbox{ for $(\mu_\omega)_x$-a.e. $y\in\E_\omega$}
\end{equation*}
which completes the proof of Theorem \ref{thm local entropy}.

\begin{proof}[Proof of Lemma \ref{lemma lower bound for local entropy}]
	Let $\epsilon_0$ be a sufficiently small number to be specified later, and let $\epsilon\in(0,\epsilon_0]$. For $r>0$ and $x\in\E_\omega$, we denote $B_\omega(x,r)=\{y\in\E_\omega:\ d(x,y)<r\}$, $V_\omega(r,\alpha)=\{x\in \E_\omega:\ B_\omega(x,r)\not\subset \alpha_\omega(x)\}$, and $V(r,\alpha)=\{(\omega,x):\ x\in V_\omega(r,\alpha)\}\in(\F\otimes\mathcal{B})\cap \E$. Since $\cap_{r>0}V_\omega(r,\alpha)=\partial \alpha_\omega$ for all $\omega\in\Omega$ and $\int\mu_\omega(\partial \alpha_\omega)d\P(\omega)=0$, it follows that
	\begin{equation*}
		\lim_{r\searrow 0}\mu_\omega(V_\omega(r,\alpha))=\mu_\omega(\partial \alpha_\omega)=0,\ \mbox{ for $\P$-a.e. $\omega\in\Omega$.}
	\end{equation*}
By the monotone convergence theorem, we have
	\begin{equation*}
		\lim_{r\searrow 0}\mu(V(r,\alpha))=\lim_{r\searrow 0}\int_\Omega\mu_\omega(V_\omega(r,\alpha))d\P(\omega)=0.
	\end{equation*}
We can choose $\tau>0$ such that $\int_\Omega\mu_\omega(V_\omega(r,\alpha))d\P(\omega)<\epsilon$ for all $r\in(0,\tau]$.
	For each $n\geq 1$, we define
		\begin{equation*}
		A_n(\omega)=\{y\in \E_\omega:\ \frac{1}{k}\sum_{i=0}^{k-1}1_{V(\tau,\alpha)}(\Theta^i(\omega,y))<2\sqrt{\epsilon},\ \forall k\geq n\}.
	\end{equation*}
	 By Birkhoff's ergodic theorem, since $\mu$ is ergodic, for $\mu$-a.e. $(\omega,y)\in\E$,
	\begin{equation*}
		\lim_{k\to+\infty}\frac{1}{k}\sum_{i=0}^{k-1}1_{V(\tau,\alpha)}(\Theta^i(\omega,y))=\mu(V(\tau,\alpha))<\epsilon.
	\end{equation*}
Thus, for $\P$-a.e. $\omega\in\Omega$, there is a smallest integer $l_0(\omega)\geq 1$ such that
	\begin{equation}\label{eq measure of A_l0}
		\mbox{$\mu_\omega(A_{n}(\omega))>1-2\sqrt{\epsilon}$ for all $n\geq l_0(\omega)$.}
	\end{equation}  We claim that $\omega\mapsto l_0(\omega)$ is measurable, since for any $j\geq 2$ the following set is measurable
	\begin{equation*}
		\begin{split}
			&\{\omega\in\Omega:\ l_0(\omega)=j\}\\
			=&\bigcap_{n\geq j}\{\omega\in\Omega:\ \mu_\omega(A_n(\omega))>1-2\sqrt{\epsilon}\}\cap\{\omega\in\Omega:\ \mu_\omega(A_{j-1}(\omega))\leq 1-2\sqrt{\epsilon}\}.
		\end{split}
	\end{equation*}
	For each $\omega\in\Omega$ for which $l_0(\omega)$ is defined, we define
	\begin{equation*}
		Q_{l_0(\omega)}(\omega):=\{x\in\E_\omega:\ (\mu_\omega)_x(A_{l_0(\omega)}(\omega))\geq 1-2\epsilon^{1/4}\}.
	\end{equation*}
Note that
	\begin{equation*}
		\begin{split}
		\mu_\omega(Q_{l_0(\omega)}(\omega)^c)\cdot 2\epsilon^{\frac{1}{4}}\leq & \int_{Q_{l_0(\omega)}(\omega)^c}(\mu_\omega)_x(A_{l_0(\omega)}(\omega)^c)d\mu_\omega\leq \int_{\E_\omega}(\mu_\omega)_x(A_{l_0(\omega)}(\omega)^c)d\mu_\omega\\
		=&\mu_\omega(A_{l_0(\omega)}(\omega)^c)\leq 2\sqrt{\epsilon}.
		\end{split}
	\end{equation*}
Thus, we have $\mu_\omega(Q_{l_0(\omega)}(\omega))>1-\epsilon^{1/4}$. We denote $Q_{l_0}:=\{(\omega,x):\ x\in Q_{l_0(\omega)}(\omega)\}$. Then $Q_{l_0}\in(\mathcal{F}\otimes\mathcal{B})\cap \E$ by the measurability of $l_0(\omega)$ and the measurability of $(\omega,x)\mapsto (\mu_\omega)_x$.
	
	By Corollary \ref{corollary random smb}, there exists a Borel measurable set $\mathcal{Y}\subset \E$ with $\mu(\mathcal{Y})=1$ such that for $(\omega,x)\in \Y$, $y\in\xi_\omega(x)\cap \Y_\omega$, one has $(\mu_\omega)_x(\xi_\omega(x)\cap \Y_\omega)=1$ and
	  \begin{equation}\label{eq corollary in proof}
		\lim_{N\to+\infty}\frac{-\log(\mu_\omega)_x\left(\vee_{k=0}^{N-1}(\Theta^{-k}\alpha)_\omega(y)\right)}{N}=H_\mu(\alpha|\mathcal{H})=c.
	\end{equation}
	Let $\L=\Y\cap Q_{l_0}$. This set is measurable, and $\mu_\omega(\L_\omega)>1-\epsilon^{1/4}$ for $\P$-a.e. $\omega\in\Omega$.
	
	From now on, fix $(\omega,x)\in \L$ and proceed with the proof of Lemma \ref{lemma lower bound for local entropy}.
 We define
	\begin{equation*}
		M_n(\omega,x)=\{y\in\xi_\omega(x)\cap \Y_\omega:\frac{-\log (\mu_\omega)_x(\vee_{i=0}^{k-1}(\Theta^{-i}\alpha)_\omega(y))}{k}\geq c-\epsilon\mbox{ for any }k\geq n\}
	\end{equation*}
By \eqref{eq corollary in proof}, we can find a smallest integer $l_1=l_1(\omega,x)\geq l_0(\omega)$ such that $(\mu_\omega)_x(M_{l_1}(\omega,x))\geq 1-\epsilon^{1/4}.$ Since $(\omega,x)\in\L$ is fixed, we write $l_1=l_1(\omega,x)$ for short. We define
	\begin{equation*}
		E(\omega,x)=A_{l_1}(\omega)\cap M_{l_1}(\omega,x).
	\end{equation*}
Since $x\in Q_{l_0(\omega)}(\omega)$ and $l_1\geq l_0(\omega)$, we have $(\mu_\omega)_x(A_{l_1}(\omega))\geq 1-2\epsilon^{1/4}$. At the same time, $(\mu_\omega)_x(M_{l_1}(\omega,x))\geq 1-\epsilon^{1/4}$ implies that $(\mu_\omega)_x(E(\omega,x))\geq 1-3\epsilon^{1/4}.$
	
	For any $n\geq 1$ and $W\in \vee_{i=0}^{n-1}(\Theta^{-i}\alpha)_\omega$, we define $\alpha_{\theta^i\omega}(F_\omega^i W)$ to be the unique atom in $\alpha_{\theta^i\omega}$ that contains $F_\omega^i W$ for each $0\leq i\leq n-1$. For $W_1,W_2\in \vee_{i=0}^{n-1}(\Theta^{-i}\alpha)_\omega$, we define the following pseudo-metric
	\begin{equation*}
		\rho_{n,\omega}(W_1,W_2)=\frac{1}{n}\#\{0\leq i\leq n-1,\ \alpha_{\theta^i\omega}(F_\omega^i W_1)\not =\alpha_{\theta^i\omega}(F_\omega^i W_2)\}.
	\end{equation*}
For $W\in\vee_{i=0}^{n-1}(\Theta^{-i}\alpha)_\omega $, we define the $2\sqrt{\epsilon}$-neighborhood of $W$ with respect to this pseudo-metric
\begin{equation}\label{eq def N n w}
	\N_{n,\omega}(W):=\{V\in \vee_{i=0}^{n-1}(\Theta^{-i}\alpha)_\omega:\ \rho_{n,\omega}(V,W)<2\sqrt{\epsilon}\}.
\end{equation}

We claim that
\begin{equation}\label{eq Ball is contained in N}
	B_\omega(y,\tau,n)\subset \bigcup\{V:\ V\in\N_{n,\omega}(\vee_{i=0}^{n-1}(\Theta^{-i}\alpha)_\omega(y)) \},\ \forall y\in E(\omega,x),\ n\geq l_1.
\end{equation}
If $z\in B_\omega(y,\tau,n)$, then for any $0\leq i<n$, either $F_\omega^i(y)$ and $F_\omega^i(z)$ belong to the same element of $\alpha_{\theta^i\omega}$, or $F_\omega^i(y)\in V_{\theta^i\omega}(\tau,\alpha)$, and when $y\in E(\omega,x)\subset A_{l_1}(\omega),$ $\frac{1}{n}\sum_{i=0}^{n-1}1_{V(\tau,\alpha)}(\Theta^i(\omega,y))<2\sqrt{\epsilon}$. Therefore, if $z\in B_\omega(y,\tau,n)$, then $\rho_{n,\omega}(\vee_{i=0}^{n-1}(\Theta^{-i}\alpha)_\omega(y),\vee_{i=0}^{n-1}(\Theta^{-i}\alpha)_\omega(z))<2\sqrt{\epsilon}$.

From the definition of \eqref{eq def N n w}, for $W\in \vee_{i=0}^{n-1}(\Theta^{-i}\alpha)_\omega$, we have
\begin{equation*}
	\#\N_{n,\omega}(W)\leq \sum_{i=0}^{m}\begin{pmatrix}
		n\\ i
	\end{pmatrix}(\#\alpha-1)^i\mbox{ for $m=\lceil 2n\sqrt{\epsilon}\rceil$.}
\end{equation*}
Stirling's formula (see (I.2) and (I.3) in \cite{Katok1980}) implies that, after choosing $\epsilon_0$ sufficiently small, there exists a constant $l_2=l_2(\epsilon,\alpha)\geq 1$ such that
\begin{equation}\label{eq number of Nnw}
	\#\N_{n,\omega}(W)\leq e^{(\Delta+\epsilon)n},\ \forall n\geq l_2, \ \forall \epsilon\in (0,\epsilon_0)
\end{equation}
where $\Delta$ is defined by \eqref{eq delta}.

Next, we choose $D(\omega,x)$ in the statement of Lemma \ref{lemma lower bound for local entropy} by removing certain ``bad sets" from $E(\omega,x)$. Set
\begin{equation*}
  \mathcal{P}_{n,\omega}:=\{V\in \vee_{i=0}^{n-1}(\Theta^{-i}\alpha)_\omega:\ (\mu_\omega)_x(V)>e^{(-c+2(\Delta+\epsilon))n}\},
\end{equation*}
and
\begin{equation*}
	\mathcal{M}_{n,\omega}:=\bigcup_{V\in \mathcal{P}_{n,\omega}}\{y\in E(\omega,x):\ \rho_{n,\omega}(\vee_{i=0}^{n-1}(\Theta^{-i}\alpha)_\omega(y),V)<2\sqrt{\epsilon}\}.
\end{equation*}
By the definition of $\mathcal{P}_{n,\omega}$, we have
\begin{equation}\label{eq number of Pnw}
	\#\mathcal{P}_{n,\omega}=\#\mathcal{P}_{n,\omega}\cdot 1\leq \sum_{V\in \mathcal{P}_{n,\omega}}(\mu_\omega)_x(V)e^{(c-2(\Delta+\epsilon))n}\leq e^{(c-2(\Delta+\epsilon))n}.
\end{equation}
To estimate the size of $\mathcal{M}_{n,\omega}$, we let
\begin{equation}\label{eq def Onw}
	\begin{split}
	\O_{n,\omega}:=&\bigcup_{W\in \mathcal{P}_{n,\omega}}\{V\in \mathcal{N}_{n,\omega}(W):\ (\mu_\omega)_x(V)<e^{(-c+\epsilon)n}\}\\
	\subset & \bigcup_{W\in \mathcal{P}_{n,\omega}} \mathcal{N}_{n,\omega}(W).
	\end{split}
\end{equation}
It follows from \eqref{eq number of Nnw} and \eqref{eq number of Pnw} that
\begin{equation}\label{eq number of Onw}
	\# 	\O_{n,\omega}\leq e^{(c-(\Delta+\epsilon))n},\ \forall n\geq l_2.
\end{equation}
By the definition of $\mathcal{M}_{n,\omega}$, for any $y\in \mathcal{M}_{n,\omega}\subset E(\omega,x)$, there exists $V\in \mathcal{P}_{n,\omega}$ such that $\rho_{n,\omega}(\vee_{i=0}^{n-1}(\Theta^{-i}\alpha)_\omega(y),V)<2\sqrt{\epsilon}$. Moreover, by the definition of $E(\omega,x)$, for any $n\geq l_1(\omega)$, one has $(\mu_\omega)_x(\vee_{i=0}^{n-1}(\Theta^{-i}\alpha)_\omega(y))\leq e^{(-c+\epsilon)n}$. Therefore,
\begin{equation*}
  \mathcal{M}_{n,\omega}\subset \bigcup_{W\in\O_{n,\omega}}W,\ \forall n\geq \max\{l_1(\omega),l_2\},
\end{equation*}
and by \eqref{eq def Onw} and \eqref{eq number of Onw},
\begin{equation*}
	(\mu_\omega)_x(\mathcal{M}_{n,\omega})\leq e^{-\Delta n},\ \forall n\geq \max\{l_1(\omega),l_2\}.
\end{equation*}
We can choose $l_2$ large enough so that whenever $n\geq l_3(\omega)=\max\{l_1(\omega), l_2\}$,
\begin{equation*}
	\sum_{n=l_3(\omega)}^{+\infty}	(\mu_\omega)_x(\mathcal{M}_{n,\omega})<\epsilon^{1/4}.
\end{equation*}
We define $D(\omega,x)=E(\omega,x)\backslash \cup_{n=l_3(\omega)}^{+\infty}	\mathcal{M}_{n,\omega}$. Then $(\mu_\omega)_x(D(\omega,x))>1-4\epsilon^{1/4}.$ Given $y\in D(\omega,x)$ and $n\geq l_3(\omega)$, for each $W\in \vee_{i=0}^{n-1}(\Theta^{-i}\alpha)_\omega$ with $\rho_{n,\omega}(W,\vee_{i=0}^{n-1}(\Theta^{-i}\alpha)_\omega(y))<2\sqrt{\epsilon}$, one must have $W\not\in \mathcal{P}_{n,\omega}$ and then $(\mu_\omega)_x(W)\leq e^{(-c+2(\Delta+\epsilon))n}$.

By \eqref{eq Ball is contained in N} and \eqref{eq number of Nnw}, for $y\in D(\omega,x)$ and $n\geq l_3(\omega)$, we have
\begin{equation*}
	(\mu_\omega)_x(B_\omega(y,\tau,n))\leq e^{(-c+3(\Delta+\epsilon))n},
\end{equation*}
which implies \eqref{eq lower bound for local entropy}. This completes the proof of Lemma \ref{lemma lower bound for local entropy}.
\end{proof}

\section{Lower bound for the Hausdorff dimension of local unstable sets}\label{section proof of thm main}

We assume $\mu\in \mathcal{I}_\P^e(\E)$ and $\log^+\lambda_1^{\delta^*}\in L^1(\mu)$ for some $\delta^*>0$. Fix a sequence $\delta_m\in (0,\delta^*]$ decreasing to $0$. Let $\E_1$ be a $\mu$-full-measure set on which \eqref{eq def fiber Lyapunov exponent} holds. Since $\mu$ is ergodic, $\chi_\mu(F)=\chi(\omega,x)<\infty$ on $\E_1$. We now complete the proof of Theorem \ref{Thm main} following the strategy described in the introduction.

To embed a metric ball in the fiber into a fiber Bowen ball, we need the following lemma, which allows the order of the limits in \eqref{eq def fiber Lyapunov exponent} to be interchanged. Lemma \ref{corollary interchangeable limits} is a direct corollary of Lemma \ref{lemma interchangeable limits}.
\begin{lemma}\label{corollary interchangeable limits}
  Assume that $\log^+\lambda_1^{\delta^*}\in L^1(\mu)$ for some $\delta^*>0$. Let $\{\delta_m\}$ be any sequence with $\delta_m\in (0,\delta^*]$ decreasing to $0$. Then the following equality holds $\mu$-a.e.:
  \begin{equation}\label{eq def fiber Lyapunov exponents on fibers}
    \chi(\omega,x)=\lim_{m\to+\infty}\Lambda^{\delta_m}(\omega,x),
  \end{equation}
where $\chi$ is defined in \eqref{eq def fiber Lyapunov exponent} and $\Lambda^{\delta_m}$ is defined in \eqref{eq def lambda delta}.
\end{lemma}
\begin{lemma}\label{lemma embed ball into bowen ball}
  Given $\gamma>\chi_\mu(F)$ and any $(\omega,y)\in\E_1$, there exists $N_1(\omega,y,\gamma)\in\mathbb{N}$ such that, for any $m\geq N_1(\omega,y,\gamma)$, there exist $N=N(\omega,y,\gamma,m)\in\mathbb{N}$ and $N_0=N_0(\omega,y,\gamma,m,N)\in\mathbb{N}$ such that
  \begin{equation*}
    B_\omega(y,e^{-\gamma(n+N_0)})\subset B_\omega(y,\delta_m,n),\ \forall n\geq N(\omega,y,\gamma,m).
  \end{equation*}
\end{lemma}
\begin{proof}[Proof of Lemma \ref{lemma embed ball into bowen ball}]
  By Lemma \ref{corollary interchangeable limits}, for any fixed $(\omega,y)\in\E_1$,
  \begin{equation*}
    \chi_\mu(F)=\chi(\omega,y)=\lim_{m\to+\infty}\lim_{n\to+\infty}\frac{1}{n}\log^+\lambda_n^{\delta_m}(\omega,y)<\infty.
  \end{equation*}
Thus, given $\gamma> \chi_\mu(F)$, there exists $N_1=N_1(\omega,y,\gamma)>0$ such that
  \begin{equation*}
    \lim_{n\to+\infty}\frac{1}{n}\log^+\lambda_n^{\delta_m}(\omega,y)<\frac{2\gamma+\chi(\omega,y)}{3},\ \forall m\geq N_1(\omega,y,\gamma).
  \end{equation*}
For $m\geq N_1(\omega,y,\gamma)$, we pick $N=N(\omega,y,\gamma,m)>N_1$ sufficiently large such that
  \begin{equation*}
    e^{-\frac{1}{2}(\gamma-\chi(\omega,y))n}<\delta_m,\ \forall n\geq N(\omega,y,\gamma,m),
  \end{equation*}
and
  \begin{equation*}
    \frac{1}{n}\log^+\lambda_n^{\delta_m}(\omega,y)<\frac{\gamma+\chi(\omega,y)}{2},\ \forall n\geq N(\omega,y,\gamma,m).
  \end{equation*}
Since $F_\omega^n$ is continuous in $x$ for any fixed $\omega$ and $n$, we can pick $N_0=N_0(\omega,y,\gamma,m,N)$ such that $B_\omega(y,e^{-\gamma(N+N_0)})\subset B_\omega(y,\delta_m,N)$. We prove by induction that $B_\omega(y,e^{-\gamma(n+N_0)})\subset B_\omega(y,\delta_m,n)$ for any $n\geq N(\omega,y,\gamma,m)$. 
 Suppose inductively that $B_{\omega}(y,e^{-\gamma(k+N_0)})\subset B_{\omega}(y,\delta_m,k)$ for some $k\geq N(\omega,y,\gamma,m)$. If $z\in B_{\omega}(y, e^{-\gamma(k+1+N_0)})$, then $\gamma>0$ gives
  $$ z\in B_{\omega}(y, e^{-\gamma(k+N_0)})\subset B_{\omega}(y,\delta_m,k).
  $$
  For $z\not= y$, the definition of $\lambda_k^{\delta_m}(\omega,y)$ gives
   \begin{equation*}
    \begin{aligned}
      d(F_\omega^k(z),F_\omega^k(y))  &\leq \lambda_k^{\delta_m}(\omega,y)d(z,y)\\
      &\leq e^{\frac{1}{2}(\gamma+\chi(\omega,y))k}\cdot e^{-\gamma(k+1+N_0)}\\
        & \leq e^{-\frac{1}{2}(\gamma-\chi(\omega,y))k}<\delta_m.
    \end{aligned}
  \end{equation*}
  The case $z=y$ is immediate. Therefore $z\in B_{\omega}(y,\delta_m, k+1)$. Induction proves that
  $$B_\omega(y,e^{-\gamma(n+N_0)})\subset B_\omega(y,\delta_m,n)$$ for all $n\geq N(\omega,y,\gamma,m)$.
\end{proof}
We continue the proof of Theorem \ref{Thm main}.
Recall that in Theorem \ref{thm local entropy} we constructed a countably generated measurable partition $\xi$ of $\E$. Let $\{(\mu_\omega)_x\}_{(\omega,x)\in\E_2}$ be the disintegration of $\mu$ over $\sigma(\xi\vee\F_\E)$, where $\E_2$ is a $\mu$-full-measure set on which this disintegration is defined. Given $\gamma>\chi_\mu(F)$, any $(\omega,y)\in\E_1$, and $m\geq N_1(\omega,y,\gamma)$, let $N=N(\omega,y,\gamma,m)\in\mathbb{N}$ and $N_0=N_0(\omega,y,\gamma,m,N)\in\mathbb{N}$ be as in Lemma \ref{lemma embed ball into bowen ball}. For $\P$-a.e. $\omega\in\Omega$, the set $(\E_1)_\omega$ has full $\mu_\omega$-measure and hence full $(\mu_\omega)_x$-measure for $\mu_\omega$-a.e. $x$.
For $(\mu_\omega)_x$-a.e. $y\in (\E_1)_\omega$, we have
\begin{equation*}
  \begin{split}
     \underline{d}_{(\mu_\omega)_x}(y)=& \liminf_{r\searrow 0}\frac{\log(\mu_\omega)_x(B_\omega(y,r))}{\log r}\geq \liminf_{n\to+\infty}\frac{-\log(\mu_\omega)_x(B_\omega(y,e^{-\gamma(n+N_0)}))}{\gamma(n+1+N_0)}\\
      \geq  &
     \liminf_{n\to+\infty}\frac{-\log(\mu_\omega)_x(B_\omega(y,\delta_m,n))}{\gamma(n+1+N_0)} = \liminf_{n\to+\infty}\frac{-\log(\mu_\omega)_x(B_\omega(y,\delta_m,n))}{\gamma \cdot n}.
  \end{split}
\end{equation*}
The first inequality follows from the following fact: For fixed $m$, put
$$
r_n:=e^{-\gamma(n+N_0)}.
$$
If $r_{n+1}\le r<r_n$, then
$$
B_\omega(y,r)\subset B_\omega(y,r_n)
$$
and
$$
-\log r\le-\log r_{n+1}
=\gamma(n+1+N_0).
$$
Consequently,
$$
\frac{-\log(\mu_\omega)_x(B_{\omega}(y,r))}{-\log r}
\ge
\frac{-\log(\mu_\omega)_x(B_{\omega}(y,r_n))}
{\gamma(n+1+N_0)}.
$$

Letting $m\to+\infty$, we obtain
\begin{equation*}
   \underline{d}_{(\mu_\omega)_x}(y) \geq \lim_{m\to+\infty} \liminf_{n\to+\infty}\frac{-\log(\mu_\omega)_x(B_\omega(y,\delta_m,n))}{\gamma \cdot n}=\frac{1}{\gamma}\cdot \underline{h}_{(\mu_\omega)_x}(F,\omega,y)=\frac{h_\mu(F)}{\gamma},
\end{equation*}
where the first equality follows from Theorem \ref{thm local entropy} (2). Since $(\mu_\omega)_x$ is concentrated on $\xi_\omega(x)\subset W_\omega^u(x,\delta)$, by the mass distribution principle, we have
\begin{equation*}
  \dim_H(W_\omega^u(x,\delta))\geq \frac{h_\mu(F)}{\gamma},\ \mbox{for } \P\mbox{-a.e. }\omega\in\Omega,\ \mu_\omega\mbox{-a.e. $x$}.
\end{equation*}
Consider the sequence $\gamma_q=\chi_{\mu}(F)+\frac1q$, where $q\in\mathbb{N}$. After intersecting the corresponding countable collection of full-measure sets, we have
\begin{equation*}
  \dim_H(W_\omega^u(x,\delta))\geq \frac{h_\mu(F)}{\chi_\mu(F)+\frac1q},\ \mbox{for every $q\in \mathbb{N}$}.
\end{equation*}
Letting $q\to+\infty$, we obtain
\begin{equation*}
  \dim_H(W_\omega^u(x,\delta))\geq \frac{h_\mu(F)}{\chi_\mu(F)},\ \mbox{for } \P\mbox{-a.e. }\omega\in\Omega,\ \mu_\omega\mbox{-a.e. $x$}.
\end{equation*}
This completes the proof of Theorem \ref{Thm main}. 

\section{Proof of Theorem \ref{theorem 4}}\label{sec thm 4}
In this section, we prove Theorem \ref{theorem 4}. Let $(\E,F)$ be a continuous bundle RDS and let $\mu\in\mathcal{I}_\P^e(\E)$. Suppose that $\log^+\lambda_1^{\delta^*}\in L^1(\mu)$ for some $\delta^*>0$. Then Lemma \ref{corollary interchangeable limits} applies, and
\begin{equation}\label{eq LE=0}
  \mbox{$\chi_\mu(F)=\chi(\omega,x)$ for $\mu$-a.e. $(\omega,x)\in\E$}
\end{equation}
holds. We assume that the conditional measures of $\mu$ have finite upper pointwise dimension, i.e.,
   \begin{equation}\label{eq upperpointdimfinite}
     \overline{d}_{\mu_\omega}(x):=\limsup_{r\searrow 0}\frac{\log \mu_\omega(B_\omega(x,r))}{\log r}<+\infty,\mbox{ for $\mu$-a.e. $(\omega,x)\in\E$}.
   \end{equation}
We now fix a typical point $(\omega,x)$ in a $\mu$-full-measure set on which \eqref{eq LE=0}, \eqref{eq upperpointdimfinite}, \eqref{eq def fiber Lyapunov exponents on fibers}, and \eqref{eq brin-katok} all hold.
For any $D>\overline{d}_{\mu_\omega}(x)$, there exists $r_D=r_D(\omega,x)\in (0,1)$ such that
\begin{equation}\label{eq rD}
  \mbox{when $\rho\in(0,r_D)$, one has $\mu_\omega(B_\omega(x,\rho))\geq \rho^D$.}
\end{equation} For any $\eta>\chi_\mu(F)$, by Lemma \ref{corollary interchangeable limits}, there exists $m_\eta=m_\eta(\omega,x)$ such that
\begin{equation}\label{}
  \mbox{ for any $m>m_\eta$, one has $\Lambda^{\delta_m}(\omega,x)\leq \eta.$ }
\end{equation}
For any $m>m_\eta$ and $\eta^\prime>\eta$, we then choose $N_{\eta^\prime,m}=N_{\eta^\prime,m}(\omega,x)$ such that
\begin{equation}\label{eq leq etaprimek}
  \mbox{ for any $k\geq N$, one has $\log^+\lambda_k^{\delta_m}(\omega,x)< \eta^\prime k$}.
\end{equation}
For any $m>m_\eta$, $\eta^\prime>\eta$, since $F_\omega^j$ is continuous on $\E_\omega$ for $j=0,...,N_{\eta^\prime,m}-1$, there exists $r_{\eta^\prime,m}=r_{\eta^\prime,m}(\omega,x)>0$ such that
\begin{equation}\label{eq force steps}
  \mbox{ $d(x,y)<r_{\eta^\prime,m}$ for $x,y\in\E_\omega$ implies $d(F_\omega^jx,F_\omega^jy)<\delta_m$ for all $0\leq j<N_{\eta^\prime,m}$.}
\end{equation}

Now for any $n> N_{\eta^\prime,m}$, we define
\begin{equation*}
 A_n^{\delta_m}= A_n^{\delta_m}(\omega,x):=\max_{N_{\eta^\prime,m}\leq k\leq n}\log^+\lambda_k^{\delta_m}(\omega,x).
\end{equation*}
By \eqref{eq leq etaprimek}, we have $ A_n^{\delta_m}(\omega,x)\leq \eta^\prime n$ for any $n\geq N_{\eta^\prime,m}$. Set
\begin{equation*}
  \rho_*=\rho_*(D,\eta^\prime,m,n):=\frac{1}{2}\min\{r_D,r_{\eta^\prime,m},\delta_me^{-A_n^{\delta_m}}\}.
\end{equation*}
We claim that $B_\omega(x,\rho_*)\subset B_\omega(x,\delta_m,n)$. Suppose, to the contrary, that there exists $y\in\E_\omega$ such that $d(x,y)<\rho_*$ but $y\not\in B_\omega(x,\delta_m,n)$. Let $1\leq k<n$ be the smallest integer such that $d(F_\omega^k x,F_\omega^ky)\geq \delta_m$. Then \eqref{eq force steps} forces $k\geq N_{\eta^\prime,m}$. Moreover, by the minimality of $k$, we have $y\in B_\omega(x,\delta_m,k)$. Therefore, $d(F_\omega^kx,F_\omega^k y)\leq \lambda_k^{\delta_m}(\omega,x)d(x,y)< e^{A_n^{\delta_m}}\cdot \delta_me^{-A_n^{\delta_m}}=\delta_m$, a contradiction.

Now, for any $n>N_{\eta^\prime,m}$, we have
\begin{equation*}
\begin{split}
  -\frac{1}{n}\log\mu_\omega(B_\omega(x,\delta_m,n))\leq  &  -\frac{1}{n}\log\mu_\omega( B_\omega(x,\rho_*))\overset{\eqref{eq rD}}{\leq} -\frac{D}{n}\log \rho_* \\
    \leq & -\frac{D}{n}\left(-\log 2+\max\{\log r_D,\log r_{\eta^\prime,m},-A_n^{\delta_m}+\log\delta_m\}\right).
\end{split}
\end{equation*}
Note that $A_n^{\delta_m}(\omega,x)\leq \eta^\prime n$ for any $n\geq N_{\eta^\prime,m}$. Therefore,
\begin{equation*}
  \limsup_{n\to+\infty}-\frac{1}{n}\log\mu_\omega(B_\omega(x,\delta_m,n))\leq D\eta^\prime.
\end{equation*}
Letting first $\eta'\searrow\eta$ and then $m\to+\infty$, we obtain
\begin{equation*}
  \lim_{\epsilon\to 0}\limsup_{n\to+\infty}-\frac{1}{n}\log\mu_\omega(B_\omega(x,\epsilon,n))=\lim_{m\to+\infty}\limsup_{n\to+\infty}-\frac{1}{n}\log\mu_\omega(B_\omega(x,\delta_m,n))\leq D\cdot\eta.
\end{equation*}
Since $\eta>\chi_\mu(F)$ is arbitrary, \eqref{eq brin-katok} gives $h_\mu(F)\leq D\cdot \chi_\mu(F)$. Since $D>\overline{d}_{\mu_\omega}(x)$ is also arbitrary, we obtain $h_\mu(F)\leq \overline{d}_{\mu_\omega}(x)\cdot \chi_\mu(F)$.

\section{Comparison with the classical maximal Lyapunov exponents}\label{section Lyapunov exponents}

In this section, we consider the case in which $X=M$ is a compact Riemannian manifold. Let $F$ be a $C^1$ RDS such that $F_\omega\in \mathrm{Diff}^1(M)$ for $\P$-a.e. $\omega$, and let $\mu\in \mathcal{I}_\P^e(\Omega\times M)$.
Assume the integrability condition $\log^+\|D_xF_\omega\|,\ \log^+\|D_xF_\omega^{-1}\|\in L^1(\Omega\times M,\mu)$.
By Oseledets' multiplicative ergodic theorem for RDSs (see \cite[Theorem 2.1]{Liu2001}), there is a $\Theta$-invariant measurable set $\triangle\subset \Omega\times M$ such that $\mu(\triangle)=1$ and, for $(\omega,x)\in\triangle$, there exist numbers depending on $(\omega,x)$,
\begin{equation*}
  -\infty<\gamma^1(\omega,x)<\cdots <\gamma^r(\omega,x)<+\infty
\end{equation*}
and an associated measurable splitting
\begin{equation*}
  T_xM=E^{1}(\omega,x)\oplus \cdots \oplus E^r(\omega,x)
\end{equation*} such that
\begin{equation*}
  \lim_{n\to+\infty}\frac{1}{n}\log \|D_xF_\omega^{\pm n}v\|=\pm\gamma^{i}(\omega,x),\mbox{ for $0\not=v\in E^i(\omega,x)$},\ 1\leq i\leq r.
\end{equation*}
\begin{theorem}\label{thm 3}
  Let $F$ be a $C^1$ RDS such that $F_\omega\in \mathrm{Diff}^1(M)$ for $\P$-a.e. $\omega$, and let $\mu\in \mathcal{I}_\P^e(\Omega\times M)$. Assume that the integrability condition $\log^+\|D_xF_\omega\|,\ \log^+\|D_xF_\omega^{-1}\|\in L^1(\mu)$ holds and that $\gamma^r>0$. Suppose that $\log^+\lambda_1^{\delta^*}\in L^1(\mu)$ for some $\delta^*>0$.  
  Let $\chi_\mu(F)=\chi(\omega,x)<\infty$ be defined on a $\mu$-full-measure set $\E_1$ as in \eqref{eq def fiber Lyapunov exponent}. Then
  \begin{equation}
    \chi(\omega,x)=\gamma^r(\omega,x),\ \mbox{for $\mu$-a.e. }(\omega,x)\in\Omega\times M.
  \end{equation}
\end{theorem}
\begin{proof}
  We prove the conclusion by showing that $\lambda_n(\omega,x)=\|D_xF_\omega^n\|$ for those $(\omega,x)\in\Omega\times M$ and $n\in \mathbb{N}$ for which $\chi(\omega,x)$ and $\gamma^r(\omega,x)$ are well-defined.

  We first show $\lambda_n(\omega,x)\leq \|D_xF_\omega^n\|$ for any fixed $(\omega,x)$ and $n\in\mathbb{N}$. For any fixed $\epsilon>0$, since $F_\omega^n\in \rm{Diff}^1(M)$, there exists a neighborhood $B_\omega(x,\delta^*(n))\subset M$ of $x$ such that
  \begin{equation*}
    \|D_zF_\omega^n\|\leq \|D_xF_\omega^n\|+\epsilon,\ \forall z\in B_\omega(x,\delta^*(n)).
  \end{equation*}
Now, for $y\in B_\omega(x,\delta,n)$ such that $\delta<\delta^*(n)$ and the $\delta$-neighborhood of $x$ in $M$ lies in a normal neighborhood of $x$, let $\Gamma:[0,1]\to M$ be the unique geodesic from $x$ to $y$, which stays inside $B_\omega(x,\delta^*(n))$. Then
  \begin{equation*}
    d(F_\omega^n(x),F_\omega^n(y))\leq \int_0^1\|D_{\Gamma(t)}F_\omega^n\|\cdot\|\Gamma^\prime(t)\|dt\leq (\|D_xF_\omega^n\|+\epsilon)\int_{0}^{1}\|\Gamma^\prime(t)\|dt=(\|D_xF_\omega^n\|+\epsilon)d(x,y).
  \end{equation*}
Taking the supremum over such $y$ and then letting $\delta\downarrow 0$ gives $\lambda_n(\omega,x)\leq \|D_xF_\omega^n\|+\epsilon$. Then we obtain $\lambda_n(\omega,x)\leq \|D_xF_\omega^n\|$ by letting $\epsilon\to 0.$

  We now show $\lambda_n(\omega,x)\geq \|D_xF_\omega^n\|$ for any fixed $(\omega,x)$ and $n\in\mathbb{N}$. Choose a unit vector $v\in T_xM$ such that $\|D_xF_\omega^n(v)\|\geq \|D_xF_\omega^n\|-\epsilon$. Define $y_t=\exp_x(tv)$. By the differentiability of $F_\omega^n$ in exponential coordinates at $x$, we have
  \begin{equation*}
    \lim_{t\to 0}\frac{d(F_\omega^n(x),F_\omega^n(y_t))}{d(x,y_t)}=\|D_xF_\omega^n(v)\|.
  \end{equation*}
For any $\delta>0$, we can choose $t$ small enough such that $y_t\in B_\omega(x,\delta,n)$ and
  \begin{equation*}
   \lambda_n^\delta(\omega,x)\geq  \frac{d(F_\omega^n(x),F_\omega^n(y_t))}{d(x,y_t)}\geq \|D_xF_\omega^n(v)\|-\epsilon\geq \|D_xF_\omega^n\|-2\epsilon.
  \end{equation*}
Taking $\delta\downarrow 0$ yields $\lambda_n(\omega,x)\geq \|D_xF_\omega^n\|-2\epsilon$. We obtain $ \lambda_n(\omega,x)=\|D_xF_\omega^n\|$ by letting $\epsilon \to 0.$
The conclusion follows.
\end{proof}
\begin{appendix}
\section{}\label{appendix}
The following lemma is Lemma 4.1 in \cite{Fengscichina2022}.
\begin{lemma}\label{lemma interchangeable limits}
Let $T$ be a measure-preserving transformation on a probability space $(X,\mathcal{X},\mu)$. Let $\{\phi_{n}^{m}\}$ be a family of nonnegative measurable functions on $X$ with $\phi_1^1\in L^1(\mu)$. Suppose that, for each $m$, $\phi^m_n$ is subadditive in $n$ with respect to $T$ and that, for each $n$, $\phi_n^m$ is decreasing in $m$. Then the following two iterated limits both converge $\mu$-a.e. and in $L^1(\mu)$, and the order of the limits can be exchanged:
\begin{equation*}
  \lim_{m\to+\infty}\lim_{n\to+\infty}\frac{1}{n}\phi_n^m=\lim_{n\to+\infty}\lim_{m\to+\infty}\frac{1}{n}\phi_n^m.
\end{equation*}
\end{lemma}
\end{appendix}
\subsection*{Acknowledgements}
This work was partially supported by the National Key R\&D Program of China (Nos. 2024YFA1013602, 2024YFA1013600),  the National Natural Science Foundation of China (Nos. 12401230,  12622113, 12471188, 12371197, 12426201), and the Jiangsu Provincial Scientific Research Center of Applied Mathematics (BK20233002).

\end{document}